\documentclass[brochure,english,12pt]{smfbourbaki}

\usepackage[T1]{fontenc}
\usepackage{lmodern,amssymb,bm}

\usepackage[english]{babel}
\usepackage{tikz}

\usepackage[utf8]{inputenc}

\usepackage[colorlinks=true, linkcolor=blue, citecolor=red,
urlcolor=blue]{hyperref}

\usepackage[
backend=bibtex,
style=authoryear, 
citestyle=authoryear-comp,
maxnames=5
]{biblatex}
\usepackage{csquotes}

\usepackage{dutchcal}

\newtheorem{question}{Question}

\renewcommand{\hat}{\widehat}

\newcommand{\N}{\mathbb N}

\newcommand{\eqdef}{\stackrel{\mathrm{def}}{=}}

\newcommand{\x}{\boldsymbol{x}}
\newcommand{\y}{\boldsymbol{y}}

\newcommand{\mb}{\mathbb}
\newcommand*\diff{\mathop{}\!\mathrm{d}}

\newcommand{\R}{\mathbb{R}}
\newcommand{\e}{\varepsilon}
\newcommand{\mc}{\mathcal}

\newcommand{\msf}{\mathsf}

\newcommand{\MM}{\mathcal{M}}
\newcommand{\NN}{\mathcal{N}}

\newcommand{\dint}{\int\!\!\!\!\int}

\newcommand*\circled[1]{\tikz[baseline=(char.base)]{
            \node[shape=circle,draw,inner sep=2pt] (char) {#1};}}

\DeclareDelimFormat{nameyeardelim}{\addcomma\space}

\DeclareNameAlias{sortname}{given-family}

\renewcommand{\bibnamedash}{\leavevmode\raise3pt\hbox to3em{\hrulefill}\space}

\AtEveryBibitem{%
  \clearfield{issn} 
  \clearfield{isbn} 
  \clearfield{doi} 

  \ifentrytype{online}{}{
    \clearfield{url}
  }
}

\renewbibmacro{in:}{%
    \ifentrytype{article}{}{\printtext{\bibstring{in}\intitlepunct}}}

\DeclareFieldFormat[article,periodical,inreference]{number}{\mkbibparens{#1}}
\DeclareFieldFormat[article,periodical,inreference]{volume}{\mkbibbold{#1}}
\renewbibmacro*{volume+number+eid}{%
    \printfield{volume}%
    \setunit*{\addthinspace}
    \printfield{number}%
    \setunit{\addcomma\space}%
    \printfield{eid}}

\DeclareFieldFormat[article,inbook,incollection]{title}{\enquote{#1}\addcomma} 

\date{Janvier 2022}
\bbkannee{74\textsuperscript{e} année, 2021--2022}  
\bbknumero{1188}                                      

\title{Average distortion embeddings, nonlinear spectral gaps, and a metric John theorem}
\subtitle{after Assaf Naor}

\author{Alexandros Eskenazis}
\address{Trinity College\\ University of Cambridge\\ Cambridge, CB2 1TQ, United Kingdom}
\address{Institut de Math\'ematiques de Jussieu\\ CNRS \& Sorbonne Universit\'e\\ Paris, 75252, France }

\email{ae466@cam.ac.uk}

\thanks{The author was supported by a Junior Research Fellowship from Trinity College, Cambridge.}

\begin{document}

\maketitle


\section{Introduction} 

\noindent {\bf Preamble.} The main purpose of this survey is to present a concise exposition of some applications of the theory of nonlinear spectral gaps which can serve as a roadmap for newcomers in the field and experts alike. Having as our main focus a result (Theorem \ref{thm:john}) of \textcite{Nao21}, we shall highlight some ideas which have played a pivotal role in recent developments and mention connections with classical geometric and algorithmic questions. The material of this paper is a mere expository repackaging of a selection of such developments and any difference in presentation is solely cosmetic.

\medskip

Let $(\MM,d_\MM)$, $(\NN,d_\NN)$ be two metric spaces and $D\in[1,\infty)$. We say that $(\MM,d_\MM)$ embeds into $(\NN,d_\NN)$ with bi-Lipschitz distortion at most $D$ if there exists a scaling factor $\sigma\in(0,\infty)$ and a map $f:\MM\to\NN$ such that
\begin{equation} \label{eq:bi-lip}
\forall \ x,y\in\MM, \qquad \sigma d_\MM(x,y) \leq d_\NN\big(f(x),f(y)\big) \leq \sigma Dd_\MM(x,y).
\end{equation}
Following \textcite{Nao21}, we say that an \emph{infinite}\footnote{The study of average distortion embeddings for \emph{finite} metric spaces goes back at least to the work of \textcite{Rab03} (see also \cite{ABN11} for various related notions).} metric space $(\MM,d_\MM)$ embeds into $(\NN,d_\NN)$ with $q$-average distortion $D$, where $q>0$, if for every Borel probability measure $\mu$ on $\MM$, there exists $\sigma=\sigma_\mu\in(0,\infty)$ \mbox{and a $\sigma D$-Lipschitz map $f=f_\mu:\MM\to\NN$ with}
\begin{equation} \label{eq:av-dist}
\int\!\!\!\!\int_{\MM\times\MM} d_\NN\big( f(x),f(y)\big)^q \,\diff\mu(x)\diff\mu(y) \geq \sigma^q \dint_{\MM\times\MM} d_\MM(x,y)^q\,\diff\mu(x)\diff\mu(y).
\end{equation}
If the target space $\NN$ is a normed space, the parameter $\sigma_\mu$ can be omitted by rescaling. 

The $\theta$-snowflake of a metric space $(\MM,d_\MM)$ is the metric space $(\MM,d_\MM^\theta)$, $\theta\in(0,1]$. The primary goal of this survey is to present a self-contained proof of the following deep embedding theorem of \textcite{Nao21} in which asymptotically optimal bounds for the quadratic average distortion (i.e.~corresponding to exponent $q=2$ in equation \eqref{eq:av-dist} above) of $\tfrac{1}{2}$-snowflakes of finite-dimensional normed spaces into the separable Hilbert space $\ell_2$ are established. The, so called, \emph{average John theorem} reads as follows.

\begin{theo}[Average John]  \label{thm:john}
There exists a universal constant $C\in(0,\infty)$ such that the $\tfrac{1}{2}$-snowflake of any finite-dimensional normed space $(X,\|\cdot\|_X)$ admits an embedding into $\ell_2$ with quadratic average distortion at most $C\sqrt{\log(\mathrm{dim}(X)+1)}$.
\end{theo}

Theorem \ref{thm:john} is a metric counterpart of a classical theorem of \textcite{Joh48}, asserting that any finite-dimensional normed space embeds into $\ell_2$ with bi-Lipschitz distortion at most $\sqrt{\mathrm{dim}(X)}$. This statement is famously optimal, e.g.~for $X=\ell_1^d$ or $X=\ell_\infty^d$, yet Naor's theorem shows that an exponential improvement of the relevant distortion is possible if one relaxes the pointwise lower bound of the bi-Lipschitz condition \eqref{eq:bi-lip} to the averaged requirement \eqref{eq:av-dist} and replaces the normed space $(X,\|\cdot\|_X)$ by its $\tfrac{1}{2}$-snowflake.  Before explaining the ideas that come into the proof of Theorem \ref{thm:john}, it is worth pointing out that both of these modifications of John's theorem are necessary in order to deduce bounds for the distortion which are subpolynomial on $\mathrm{dim}(X)$. In fact, the average John theorem is optimal in three distinct ways. 

\smallskip

$\bullet$ If one is interested in bi-Lipschitz embeddings of snowflakes of normed spaces $X$ into $\ell_2$ in lieu of average distortion embeddings, then the relevant distortion has to depend polynomially on $\mathrm{dim}(X)$. Indeed, in \textcite[Lemma~2]{Nao21}, it is shown that the bi-Lipschitz distortion required to embed the $\theta$-snowflake of $\ell_\infty^d$ into $\ell_2$ is at least a constant multiple of $d^{\theta/2}$. The proof relies on metric cotype.

\smallskip

$\bullet$ The exponent $\tfrac{1}{2}$ is the least amount of snowflaking that one needs to perform in order to obtain embeddings whose quadratic average distortion depends subpolynomially on $\mathrm{dim}(X)$. More specifically, in \textcite[Lemma~13]{Nao21} it is shown that for any $\e\in(0,\tfrac{1}{2}]$, the quadratic average distortion required to embed the $(\tfrac{1}{2}+\e)$-snowflake of $\ell_1^d$ into $\ell_2$ is at least a constant multiple of $d^\e$. The proof relies on Enflo type.

\smallskip

$\bullet$ Finally, $\sqrt{\log\mathrm{dim}(X)}$ is the asymptotically optimal bound for the quadratic average distortion required to embed the $\tfrac{1}{2}$-snowflake of an arbitrary finite-dimensional space $X$ into $\ell_2$. This will be further explained (for $X=\ell_\infty^d$) in Remark \ref{rem:sharp3} below.

\smallskip

In the rest of the introduction, we shall describe the strategy of the proof of the average John theorem and introduce the necessary background. 


\subsection{Nonlinear spectral gaps} \label{subsec:nsg}

Let $\triangle^{n-1}=\{(\pi_1,\ldots,\pi_n)\in[0,1]^n: \ \sum_{i=1}^n \pi_i=1\}$ be the $n$-dimensional standard simplex. Consider a (row)-stochastic matrix $A=(a_{ij})_{i,j=1}^n \in M_n(\R)$, that is, a matrix for which $(a_{i1},\ldots,a_{in})\in\triangle^{n-1}$ for every $i\in\{1,\ldots,n\}$. Given a vector $\pi=(\pi_1,\ldots,\pi_n)\in\triangle^{n-1}$, we say that the matrix $A$ is $\pi$-reversible if $\pi_i a_{ij} = \pi_j a_{ji}$ for every $i,j\in\{1,\ldots,n\}$. These objects admit a classical probabilistic interpretation. Consider the discrete-time homogeneous Markov chain $(X_t)_{t\geq0}$ on the state space $\{1,\ldots,n\}$ with transition probabilities given by
\begin{equation}
\forall \ i,j\in\{1,\ldots,n\}, \qquad \mb{P}\{X_{t+1}=j \ | \ X_t=i\} = a_{ij},
\end{equation}
where $t\geq0$. If the transition matrix $A$ is $\pi$-reversible, then $\pi$ is also a stationary distribution for the process $(X_t)_{t\geq0}$, that is, if $X_0$ is distributed according to $\pi$ then so is $X_t$ for any $t\geq1$. This is expressed algebraically by the matrix identity $\pi A=\pi$, where $\pi$ is thought of as a row-vector. In the probabilistic framework above, reversibility simply means that the Markov process is invariant under time reversal in the sense that $(X_0,X_1,\ldots,X_T)$ has the same joint distribution as $(X_T,X_{T-1},\ldots,X_0)$ for any $T\in\N$.

Consider the Hilbert space $L_2(\pi) = (\R^n,\|\cdot\|_{L_2(\pi)})$ whose (semi-)norm is given by
\begin{equation}
\forall \ x=(x_1,\ldots,x_n)\in\R^n, \qquad \|x\|_{L_2(\pi)} = \Big( \sum_{i=1}^n \pi_i x_i^2\Big)^{\frac{1}{2}}.
\end{equation}
Analytically, the stochastic matrix $A$ is $\pi$-reversible if and only if it defines a self-adjoint contraction on $L_2(\pi)$ with real eigenvalues which we shall denote by $1=\lambda_1(A)\geq\lambda_2(A)\geq\cdots\geq\lambda_n(A)\geq-1$. The {spectral gap} of $A$ is the algebraic quantity $1-\lambda_2(A)$ which is known to encode important combinatorial properties of the matrix. It is a simple linear algebra exercise to show that the reciprocal $\gamma(A) \eqdef (1-\lambda_2(A))^{-1}$ of the spectral gap is the least constant $\gamma\in(0,\infty]$ for which the inequality
\begin{equation} \label{eq:spegap}
\forall \  x_1,\ldots,x_n\in\ell_2, \qquad \sum_{i,j=1}^n \pi_i\pi_j \|x_i-x_j\|_{\ell_2}^2 \leq \gamma \sum_{i,j=1}^n \pi_i a_{ij} \|x_i-x_j\|_{\ell_2}^2
\end{equation}
holds true. It is a well-known consequence of Cheeger's inequality  (see, e.g., \textcite{DSV03}) that upper bounds on $\gamma(A)$ are equivalent to good expansion properties of the underlying weighted graph defined by $A$.

The above analytic characterization of a spectral gap as an optimal constant in a functional inequality was the starting point for the theory of \emph{nonlinear spectral gaps}, of which Theorem \ref{thm:john} is the latest application. Let $(\MM,d_\MM)$ be a metric space and $p\in(0,\infty)$. If $\pi\in\triangle^{n-1}$ and $A$ is a $\pi$-reversible stochastic matrix, the spectral gap of $A$ with respect to $d_\MM^p$, denoted by $\gamma(A,d_\MM^p)$, is the least $\gamma\in(0,\infty]$ such that
\begin{equation} \label{eq:nsg}
\forall  \ x_1,\ldots,x_n\in\MM, \qquad \sum_{i,j=1}^n\pi_i\pi_j d_\MM(x_i,x_j)^p \leq \gamma \sum_{i,j=1}^n \pi_i a_{ij} d_\MM(x_i,x_j)^p.
\end{equation}
If the metric $d_\MM$ is inherited by a norm $\|\cdot\|$, we will denote $\gamma(A,d_\MM^p)$ by $\gamma(A,\|\cdot\|^p)$. As explained in \textcite{MN14}, unless $\MM$ is a singleton, if $\gamma(A,d_\MM^p)$ is finite then $\lambda_2(A)$ is bounded away from 1 by a positive quantity depending only on $\gamma(A,d_\MM^p)$. On the other hand, obtaining sensible upper bounds for $\gamma(A,d_\MM^p)$ in terms of the usual spectral gap $1-\lambda_2(A)$ is a notoriously hard task even for very structured metric spaces $(\MM,d_\MM)$. This difficulty reflects the fact that nonlinear spectral gap inequalities \eqref{eq:nsg} capture delicate interactions of spectral properties of the matrix $A$ and geometric characteristics of the underlying metric space $(\MM,d_\MM)$.

The study of nonlinear spectral gap inequalities \eqref{eq:nsg} has led to very fruitful investigations which have been impactful in various areas of mathematics and theoretical computer science such as metric geometry, geometric group theory, operator algebras, Alexandrov geometry and approximation algorithms. We refer, for instance, to the works of \textcite{Mat97}, \textcite{Gro03}, \textcite{Laf08,Laf09}, \textcite{Pis10}, \textcite{NS11}, \textcite{Kon12}, \textcite{MN13,MN14,MN15}, \textcite{Mim15}, \textcite{Nao14, Nao17, Nao21, ANNRW18a, ANNRW18b} (see also Section \ref{sec:applications} below for a high-level exposition of some of those). The pertinence of nonlinear spectral gaps to the study of average distortion embeddings into normed spaces and Theorem \ref{thm:john} stems from an important duality \mbox{principle which was discovered by \textcite{Nao14} and which we shall now describe.}


\subsection{Duality} \label{subsec:dual}

Fix $\pi\in\triangle^{n-1}$ and a $\pi$-reversible stochastic matrix $A\in M_n(\R)$. Let $(\MM, d_\MM)$ be a metric space, $(Y,\|\cdot\|_Y)$ be a normed space and assume that the $\theta$-snowflake of $\MM$ embeds into $Y$ with $q$-average distortion $D\in[1,\infty)$. Then, for $x_1,\ldots,x_n\in\MM$, there exist $y_1,\ldots,y_n\in Y$ such that $\|y_i-y_j\|_Y\leq Dd_\MM(x_i, x_j)^{\theta}$ for every $i,j\in\{1,\ldots,n\}$ and
\begin{equation} \label{eq:comp1}
\sum_{i,j=1}^n \pi_i \pi_j \|y_i-y_j\|_Y^q \geq \sum_{i,j=1}^n \pi_i\pi_j d_\MM(x_i,x_j)^{\theta q}.
\end{equation}
Therefore, we have
\begin{equation*}
\sum_{i,j=1}^n \pi_i\pi_j d_\MM(x_i,x_j)^{\theta q} \!\stackrel{\eqref{eq:comp1}}{\leq}\!\gamma(A,\|\cdot\|_Y^q)\sum_{i,j=1}^n \pi_i a_{ij} \|y_i-y_j\|_Y^q
 \leq \!D^q\gamma(A,\|\cdot\|_Y^q)\!\sum_{i,j=1}^n \pi_i a_{ij} d_\MM(x_i,x_j)^{\theta q}
\end{equation*}
which implies that $\gamma(A,d_\MM^{\theta q})\leq D^q \gamma(A,\|\cdot\|_Y^q)$. Moreover\footnote{As usual, we denote by $\ell_q(Y) = \big\{y=(y_n)_{n\geq1}\in Y^\N: \ \|y\|_{\ell_q(Y)} \eqdef \big(\sum_{n\geq1} \|y_n\|_Y^q\big)^{1/q} <\infty \big\}$.}, as tensorization gives the identity $\gamma(A,\|\cdot\|_Y^q) = \gamma(A,\|\cdot\|_{\ell_q(Y)}^q)$ and $\gamma(A,\|\cdot\|_W^q)$ is only determined by the finite-dimensional structure of $W$, the above simple argument shows that if the $\theta$-snowflake of $\MM$ embeds with $q$-average distortion $D\in[1,\infty)$ into any Banach space $Z$ which is finitely representable in $\ell_q(Y)$, then $\gamma(A,d_\MM^{\theta q})\leq D^q \gamma(A,\|\cdot\|_Y^q)$ for any $\pi$-reversible stochastic matrix $A\in M_n(\R)$. The first important step towards Theorem \ref{thm:john} is the following \mbox{striking converse to this implication, proven by \textcite[Theorem~1.3]{Nao14}.}

\begin{theo} [Naor's duality principle] \label{thm:duality}
Suppose that $q,D\in[1,\infty)$ and $\theta\in(0,1]$. Let $(\MM,d_\MM)$ be a metric space and $(Y,\|\cdot\|_Y)$ be a Banach space such that for every $n\in\N$ and $\pi\in\triangle^{n-1}$, every $\pi$-reversible stochastic matrix $A\in M_n(\R)$ satisfies
\begin{equation} \label{eq:d-assume}
\gamma(A,d_\MM^{\theta q}) \leq D^q \gamma(A,\|\cdot\|_Y^q).\end{equation}
Then, for any $\e>0$ the $\theta$-snowflake of $\MM$ embeds into some ultrapower\footnote{We refer to \textcite{Hei80} for background on ultraproducts of Banach spaces. For the purposes of this discussion it suffices to say that an ultrapower $Z^\mathcal{U}$ of a Banach space $Z$ is a Banach space containing $Z$ with various compactness properties such that any finite-dimensional subspace of $Z^\mathcal{U}$ embeds into $Z$ with distortion $1+\e$ for any $\e>0$.} of $\ell_q(Y)$ with $q$-average distortion at most $D+\e$.
\end{theo}

We emphasize that Theorem \ref{thm:duality} is an \emph{existential} result whose proof does not shed any light on any additional properties of the average distortion embeddings at hand. Its proof consists of an elegant Hahn--Banach separation argument which we shall present in Section \ref{sec:duality}. In the setting of the average John theorem, the metric space $\MM$ is a finite-dimensional normed space $(X,\|\cdot\|_X)$, $Y$ is the Hilbert space $\ell_2$, $q=2$ and $\theta=\tfrac{1}{2}$. As any ultrapower of $\ell_2$ is itself a Hilbert space (see \cite{Hei80}), Naor's duality theorem shows that the embedding statement of Theorem \ref{thm:john} is equivalent to the following comparison estimate for nonlinear spectral gaps.

\begin{theo} \label{thm:estimate}
Let $(X,\|\cdot\|_X)$ be a finite-dimensional normed space. Then, for every $n\in\N$ and $\pi\in\triangle^{n-1}$, every $\pi$-reversible stochastic matrix $A\in M_n(\R)$ satisfies
\begin{equation} \label{eq:estimate}
\gamma(A,\|\cdot\|_X) \leq \frac{C\log(\mathrm{dim}(X)+1)}{1-\lambda_2(A)},
\end{equation}
where $C\in(0,\infty)$ is a universal constant.
\end{theo}

Theorem \ref{thm:estimate} has implicitly appeared as a special case of a much more general result concerning nonlinear spectral gaps of complex interpolation spaces \parencite[Theorem~25]{Nao21}. This family of substantially stronger nonlinear spectral gap inequalities can be used to prove (via Theorem \ref{thm:duality}) the existence of refined average distortion embeddings of snowflakes of Banach spaces which are not captured by Theorem \ref{thm:john}. This task is undertaken in great detail in \textcite{Nao21}, yet most of these results go beyond the scope of the present survey. In Section \ref{sec:estimate}, we shall present a self-contained proof of Theorem \ref{thm:estimate} which completely avoids the complex interpolation machinery of \textcite{Nao21} and is a modification of an argument which appeared in \textcite[Section~5]{Nao18}. In Section \ref{sec:beyond}, we shall present some extensions and refinements of Theorems \ref{thm:john} and \ref{thm:estimate} and highlight some key ideas from their proofs in \textcite{Nao21}.


\subsection{Extrapolation} \label{subsec:holder}

As explained above, the forthcoming proof of Theorem \ref{thm:estimate} does not rely on any sophisticated analytic machinery beyond elementary spectral properties of matrices. We will however use the following extrapolation principle for Poincar\'e inequalities.

\begin{prop} \label{prop:extrapolation}
For every $1\leq p\leq q<\infty$ there exist $c(p,q), C(p,q)\in(0,\infty)$ such that the following conclusion holds. For every normed space $(X,\|\cdot\|_X)$, every $n\in\N$, $\pi\in\triangle^{n-1}$ and every $\pi$-reversible stochastic matrix $B\in M_n(\R)$, we have
\begin{equation} \label{eq:extrapolation}
c(p,q)\cdot \gamma(B,\|\cdot\|_X^q)^{\frac{p}{q}} \leq \gamma(B,\|\cdot\|_X^p) \leq C(p,q)\cdot \gamma(B,\|\cdot\|_X^q).
\end{equation}
\end{prop}

Proposition \ref{prop:extrapolation} is the vector-valued version (due to \textcite{Che16, dLdlS21}) of the extrapolation principle for Poincar\'e inequalities \parencite{Mat97}. In Section \ref{sec:snowflake}, we shall also discuss a strengthening of Proposition \ref{prop:extrapolation} and its relation to a long-standing problem in the nonlinear theory of Banach spaces.


\subsection{Historical discussion} \label{subsec:history}

Motivated by a classical theorem of \textcite{Rib76} and kickstarted by \textcite{Bou86}, the \emph{Ribe program} is a vast research program in metric geometry which aims to uncover deep structural analogies between the local theory of normed spaces and (nonlinear) metric spaces. In the nearly four decades that lapsed since Bourgain's formalization of its objectives, the Ribe program has been an extraordinary source of surprising phenomena which arise when one studies metric spaces through the lens of Banach space theory and, vice versa, when one considers normed spaces as objects in the metric category. Numerous such key insights obtained in the last two decades originate in works of Naor and his collaborators. We refer to the surveys of \textcite{Kal08}, \textcite{Nao12, Nao18}, \textcite{Bal13}, \textcite{BJ16} and \textcite{God17} and to the monograph of \textcite{Ost13} for a snapshot of some of these advances and their applications to other areas \mbox{of mathematics and theoretical computer science.}

Theorem \ref{thm:john} is a prime example of a result conceptually belonging in the Ribe program for multiple reasons. Firstly, the statement of the theorem contains a highly nonlinear operation (snowflaking) performed on a norm and the desired embedding itself is not realized by a linear operator despite the fact that both the source and the target space are linear. Moreover, as already mentioned, the proof of Theorem \ref{thm:john} relies on the theory of nonlinear spectral gaps, a large part of which has been developed in the context of the Ribe program (see the discussion on expanders with respect to Banach spaces in Section \ref{sec:super-expanders} below). Finally, as discussed in \textcite[Section~1.4]{Nao21}, Naor's initial interest in this research direction stemmed from a question regarding the embeddability of expanders into low-dimensional normed spaces raised by \textcite{ANNRW17} in the context of the approximate nearest neighbor search problem. A negative answer to this question (see Theorem \ref{thm:expanders} below) by \textcite{Nao17,Nao21} which follows easily from Theorem \ref{thm:john} shall be explained in detail in Section \ref{sec:expanders}. Theorem \ref{thm:expanders} also provides a new negative answer to an old question of \textcite{JL84} who asked whether every $n$-point metric space admits a bi-Lipschitz embedding with constant distortion into a $d$-dimensional normed space, where $d=O(\log n)$. This question had previously been answered negatively by \textcite{AdRRP92} for small distortions and \textcite{Mat96} in general. Naor's works provide a novel and more robust approach to this problem as they highlight a specific criterion (spectral gap) which implies the intrinsic high-dimensionality of the metric space at hand. Johnson and Lindenstrauss raised this question as a step towards finding a metric version of the aforementioned classical theorem of \textcite{Joh48}. A deep and impactful nonlinear John theorem was discovered via a completely different route in the influential work of \textcite{Bou85}. Quite surprisingly, Theorem \ref{thm:john}, which answers negatively the question of Johnson and Lindenstrauss, is itself a metric version of John's theorem.


\medskip

\noindent {\bf Structure of the paper.} In Sections \ref{sec:duality} and \ref{sec:snowflake} we present the proofs of Theorem \ref{thm:duality} and Proposition \ref{prop:extrapolation} respectively. In Section \ref{sec:estimate} we use Proposition \ref{prop:extrapolation} to prove Theorem \ref{thm:estimate} which, combined with Theorem \ref{thm:duality}, completes the proof of Theorem \ref{thm:john}. In Section~\ref{sec:beyond} we present some refinements of Theorems \ref{thm:john} and \ref{thm:estimate} from \textcite{Nao21} and highlight key ideas used in their proofs. Finally, Section \ref{sec:applications} contains a high-level account of further geometric and algorithmic applications of the theory of nonlinear spectral gaps.

\medskip

\noindent {\bf Asymptotic notation.} In what follows we use the convention that for $a,b\in[0,\infty]$ the notation $a\gtrsim b$ (respectively $a\lesssim b$) means that there exists a universal constant $c\in(0,\infty)$ such that $a\geq cb$ (respectively $a\leq cb$). The notations $\lesssim_\xi$ and $\gtrsim_\chi$ mean that the implicit constant $c$ depends on $\xi$ and $\chi$  respectively.

\medskip

\noindent{\bf Acknowledgements.} I am very grateful to Florent Baudier, Manor Mendel and Assaf Naor for helpful discussions and constructive feedback.


\section{Duality and average distortion} \label{sec:duality}

In this section we present the proof of Naor's duality Theorem \ref{thm:duality}. Despite the fact that the theorem is stated for an \emph{arbitrary} metric space $(\MM,d_\MM)$, the crux of the argument is the following special case in which $\MM$ is assumed to be \emph{finite}. The general case follows by a (standard yet lengthy) discretization and compactness argument which can be found in \textcite[Section~7]{Nao21}. The finitary version stated below was proven in the case that $\pi$ is the normalized counting measure in \textcite[Theorem~1.3]{Nao14}, where it is said that the argument is inspired by the proof of \textcite[Lemma~1.1]{Bal92}.

\begin{theo}[Naor's duality -- finitary version] \label{thm:dual-fin}
Suppose that $q,D\in[1,\infty)$, $\theta\in(0,1]$, $n\in\N$ and fix $\pi\in\mathrm{int}(\triangle^{n-1})$. Let $\MM=(\{x_1,\ldots,x_n\},d_\MM)$ be a metric space and $(Y,\|\cdot\|_Y)$ be a Banach space such that every $\pi$-reversible stochastic matrix $A\in M_n(\R)$
\begin{equation} \label{eq:dual-fin}
\gamma(A,d_\MM^{\theta q}) \leq D^q \gamma(A,\|\cdot\|_Y^q).
\end{equation}
Then, for any $\e>0$ there exists $m\in\N$ and a function $f=f_\e:(\MM,d_\MM^\theta)\to\ell_q^m(Y)$ which is $(D+\e)$-Lipschitz satisfying the condition
\begin{equation}
\sum_{i,j=1}^n \pi_i \pi_j \|f(x_i)-f(x_j)\|_{\ell_q^m(Y)}^q \geq \sum_{i,j=1}^n \pi_i\pi_j d_\MM(x_i,x_j)^{\theta q}.
\end{equation}
\end{theo}

\begin{proof}
It clearly suffices to assume that $\theta=1$ as otherwise we can simply apply the same result to the $\theta$-snowflake of $\MM$. Let $\mc{C}\subseteq M_n(\R)$ be the class of all symmetric $n\times n$ matrices $(c_{ij})$ for which there exist $y_1,\ldots,y_n\in Y$, not all of which are equal, with
\begin{equation}
\forall \ i,j\in\{1,\ldots,n\}, \qquad c_{ij} = \frac{\sum_{r,s=1}^n \pi_r\pi_s d_\MM(x_r, x_s)^q}{\sum_{r,s=1}^n \pi_r\pi_s \|y_r-y_s\|_Y^q} \cdot \|y_i-y_j\|_Y^q.
\end{equation} 
Moreover, let $\mc{P}\subseteq M_n(\R)$ be the class of all symmetric $n\times n$ matrices with nonnegative entries and vanishing diagonal and consider the convex hull $\mc{Q} \eqdef \mathrm{conv}(\mc{C}+\mc{P})$.

Let $T=(t_{ij})_{i,j=1}^n$ be the $n\times n$ matrix with entries given by $t_{ij}\eqdef (D+\e)^q d_\MM(x_i,x_j)^q$ for $i,j\in\{1,\ldots,n\}$. We shall prove that $T\in\mc{Q}$. Suppose that this is not the case. Then, by the Hahn--Banach separation theorem, there exists a nonzero symmetric $n\times n$ matrix $H=(h_{ij})_{i,j=1}^n$ with vanishing diagonal such that
\begin{equation} \label{eq:hb}
\inf_{(b_{ij})_{i,j=1}^n\in\mc{Q}} \sum_{i,j=1}^n h_{ij} b_{ij} \geq (D+\e)^q \sum_{i,j=1}^n h_{ij} d_\MM(x_i,x_j)^q.
\end{equation}
Since $\mc{Q}$ contains a translate of $\mc{P}$, choosing $(b_{ij})_{i,j=1}^n\in\mc{P}$ whose only nonzero entries are those indexed by $(k,\ell)$ and $(\ell,k)$, where $k\neq\ell$, we deduce that $h_{ij}\geq0$ for every $i,j\in\{1,\ldots,n\}$. Moreover, as \mbox{$\pi_i\neq0$ for every $i\in\{1,\ldots,n\}$, we can define the parameter}
\begin{equation}
\sigma \eqdef \max_{i\in\{1,\ldots,n\}} \frac{1}{\pi_i} \sum_{r\neq i} h_{ir} \in (0,\infty)
\end{equation}
and consider the matrix $A=(a_{ij})_{i,j=1}^n$ whose entries are given by
\begin{equation}
\forall \ i,j\in\{1,\ldots,n\}, \qquad a_{ij} \eqdef \begin{cases} \frac{h_{ij}}{\sigma \pi_i}, & \mbox{if } i\neq j \\ 1- \frac{1}{\sigma\pi_i} \sum_{r\neq i} h_{ir}, & \mbox{if } i=j \end{cases}.
\end{equation}
By construction, $A$ is a $\pi$-reversible stochastic matrix as the choice of $\sigma$ guarantees that its entries are nonnegative. Moreover, inequality \eqref{eq:hb} \mbox{can be equivalently rewritten as}
\begin{equation} \label{eq:in-terms-of-a}
\inf_{(b_{ij})_{i,j=1}^n \in \mc{Q}} \sum_{i,j=1}^n \pi_i a_{ij} b_{ij} \geq (D+\e)^q \sum_{i,j=1}^n \pi_i a_{ij} d_\MM(x_i,x_j)^q.
\end{equation}
Combining \eqref{eq:in-terms-of-a} with the definition \eqref{eq:nsg} of nonlinear spectral gaps, we deduce that
\begin{equation} \label{eq:du1}
\inf_{(b_{ij})_{i,j=1}^n \in \mc{Q}} \sum_{i,j=1}^n \pi_i a_{ij} b_{ij} \geq \frac{(D+\e)^q}{\gamma(A,d_\MM^q)} \sum_{i,j=1}^n \pi_i\pi_j d_\MM(x_i,x_j)^q
\end{equation}
On the other hand, since $\mc{C}\subseteq\mc{Q}$, we have
\begin{equation} \label{eq:du2}
\begin{split}
\inf_{(b_{ij})_{i,j=1}^n \in \mc{Q}} \sum_{i,j=1}^n \pi_i a_{ij} b_{ij} & \leq \inf_{(b_{ij})_{i,j=1}^n \in \mc{C}} \sum_{i,j=1}^n \pi_i a_{ij} b_{ij} 
\\& = \inf_{y_1,\ldots,y_n\in Y} \frac{\sum_{r,s=1}^n \pi_r\pi_s d_\MM(x_r, x_s)^q}{\sum_{r,s=1}^n \pi_r\pi_s \|y_r-y_s\|_Y^q} \sum_{i,j=1}^n \pi_i a_{ij} \|y_i-y_j\|_Y^q
\\ & = \frac{1}{\gamma(A,\|\cdot\|_Y^q)} \cdot \sum_{r,s=1}^n \pi_r\pi_s d_\MM(x_r,x_s)^q.
\end{split}
\end{equation}
Combining \eqref{eq:du1}, \eqref{eq:du2} and rearranging, we deduce that $\gamma(A,d_\MM^q)\geq (D+\e)^q \gamma(A,\|\cdot\|_Y^q)$ which contradicts the assumption \eqref{eq:dual-fin}, thus proving that $T\in\mc{Q}$.

Since $T\in\mc{Q}$ and all matrices in $\mc{P}$ have nonnegative entries, we deduce that there exists $m\in\N$, $(\mu_1,\ldots,\mu_m)\in\triangle^{m-1}$ and $n$-tuples of points $\{y_1(k),\ldots,y_n(k)\} \subset Y$ not all of which are equal for each $k\in\{1,\ldots,m\}$ such that
\begin{equation} \label{eq:from-Q}
(D+\e)^q d_\MM(x_i,x_j)^q \geq \sum_{k=1}^m \underbrace{\mu_k \ \frac{\sum_{r,s=1}^n \pi_r\pi_s d_\MM(x_r, x_s)^q}{\sum_{r,s=1}^n \pi_r\pi_s \|y_r(k)-y_s(k)\|_Y^q}}_{w_k} \cdot \|y_i(k)-y_j(k)\|_Y^q,
\end{equation}
for every $i,j\in\{1,\ldots,n\}$. Consider the mapping $f:\MM\to\ell_q^m(Y)$ given by
\begin{equation}
\forall \ i\in\{1,\ldots,n\}, \qquad f(x_i) =   \big(w_1^{1/q} y_i(1),\ldots,w_m^{1/q} y_i(m)\big).
\end{equation}
Then, for $i,j\in\{1,\ldots,n\}$ we have
\begin{equation}
\begin{split}
\|f(x_i)-f(x_j)\|_{\ell_q^m(Y)} = \Big( \sum_{k=1}^m w_k \|y_i(k)-y_j(k)\|_Y^q \Big)^{\frac{1}{q}} \stackrel{\eqref{eq:from-Q}}{\leq} (D+\e) d_\MM(x_i,x_j),
\end{split}
\end{equation}
which is equivalent to $\|f\|_{\mathrm{Lip}}\leq D+\e$. Finally,
\begin{equation}
\begin{split}
\sum_{i,j=1}^n \pi_i \pi_j &\|f(x_i)-f(x_j)\|_{\ell_q^m(Y)}^q = \sum_{i,j=1}^n \pi_i \pi_j \sum_{k=1}^n w_k \|y_i(k)-y_j(k)\|_Y^q
\\ & \stackrel{\eqref{eq:from-Q}}{=} \sum_{i,j=1}^n \pi_i \pi_j \sum_{k=1}^m \mu_k \ \frac{\sum_{r,s=1}^n \pi_r\pi_s d_\MM(x_r, x_s)^q}{\sum_{r,s=1}^n \pi_r\pi_s \|y_r(k)-y_s(k)\|_Y^q} \cdot \|y_i(k)-y_j(k)\|_Y^q
\\ & = \sum_{r,s=1}^n \pi_r\pi_s d_\MM(x_r,x_s)^q
\end{split},
\end{equation}
which proves the average lower bound and completes the proof.
\end{proof}

\begin{rema}
It is worth emphasizing that Naor's Theorem \ref{thm:dual-fin} is an important addition to a long list of results in which the existence of a map with favorable metric properties is proven using duality or by exploiting the cone structure of  $\ell_p$-distance matrices. We refer, for instance, to the works of \textcite{Sch38}, \textcite{BDCK66}, \textcite{Kri65}, \textcite{WW75} on isometric embeddings, \textcite{Mau74} on factorization theory, \textcite[Proposition~15.5.2]{Mat02} on bi-Lipschitz embeddings, \textcite{Bal90}, \textcite{Esk21} on metric dimension reduction and \textcite{Bal92}, \textcite{MN13} on extensions of Lipschitz mappings.
\end{rema}


\section{Extrapolation and snowflake embeddings} \label{sec:snowflake}

In this section we present the proof of Proposition \ref{prop:extrapolation}. The argument relies on some elementary properties of the vector-valued Mazur map \parencite{Maz29}. If $(\Omega,\mu)$ is a measure space, $(X,\|\cdot\|_X)$ is a normed space and $p,q\in[1,\infty)$, consider the map $\msf{M}_{p,q}:L_p(\mu;X)\to L_q(\mu;X)$ whose action on $f\in L_p(\mu;X)$ is given by
\begin{equation} \label{eq:mazur-map}
\forall \ \omega\in\Omega,\qquad (\msf{M}_{p,q}f)(\omega) = \frac{f(\omega)}{\|f(\omega)\|_X^{1-\frac{p}{q}}}
\end{equation}
when $f(\omega)\neq 0$ and $(\msf{M}_{p,q}f)(\omega)=0$ when $f(\omega)=0$. We will use the following lemma.

\begin{lemm} \label{lem:mazur}
Let $p,q\in[1,\infty)$. For any normed space $(X,\|\cdot\|_X)$ and any functions $f,g:\Omega\to X$ with $\max\{\|f\|_{L_p(\mu;X)},\|g\|_{L_p(\mu;X)}\} \leq 1$, we have
\begin{equation} \label{eq:mazur}
\big\| \msf{M}_{p,q} f - \msf{M}_{p,q} g\big\|_{L_q(\mu;X)} \lesssim_{p,q} \big\|f-g\big\|_{L_p(\mu;X)}^{\min\{\frac{p}{q},1\}}.
\end{equation}
\end{lemm}

\begin{proof}
The scalar-valued case $X=\R$ of the proposition is classical\footnote{Since $M_{p,q}h = |h|^{p/q} \mathrm{sign}(h)$, the scalar case is a consequence of the pointwise inequalities
\begin{equation*}
|\mathrm{sign}(\alpha) |\alpha|^\omega - \mathrm{sign}(\beta) |\beta|^\omega| \leq c_1(\omega) |\alpha-\beta|^\omega
\end{equation*}
and
\begin{equation*}
|\mathrm{sign}(\alpha) |\alpha|^{\frac{1}{\omega}} - \mathrm{sign}(\beta) |\beta|^{\frac{1}{\omega}}| \leq c_2(\omega) \max\{|\alpha|,|\beta|\}^{\frac{1}{\omega}-1} |\alpha-\beta|
\end{equation*}
which are valid for every $\alpha,\beta\in\R$ and $\omega\in(0,1]$.}
and can be found in \textcite[Section~9.1]{BL00}. Consider two functions $\theta, \phi:\Omega \to \{x\in X: \ \|x\|_X=1\}$ such that $f=\|f\|_X\theta$ and $g=\|g\|_X\phi$. Then, we have $\msf{M}_{p,q}f = \|f\|_X^{p/q} \theta$ and $\msf{M}_{p,q}g=\|g\|_X^{p/q} \phi$ which imply the inequality
\begin{equation}
\begin{split}
\|\msf{M}_{p,q}f - \msf{M}_{p,q}g\|_{L_q(\mu;X)} & = \big\| \|f\|_X^{\frac{p}{q}} \theta - \|g\|_X^{\frac{p}{q}}\phi\big\|_{L_q(\mu;X)}
\\ & \leq \underbrace{\big\| \|f\|_X^{\frac{p}{q}} - \|g\|_X^{\frac{p}{q}}\big\|_{L_q(\mu;\R)}}_{{\tiny \circled{1}}} + \underbrace{\big\|\|g\|_X^{\frac{p}{q}}(\theta-\phi)\big\|_{L_q(\mu;X)}}_{{\tiny \circled{2}}}.
\end{split}
\end{equation}
By the scalar-valued version of \eqref{eq:mazur} applied to $\|f\|_X$ and $\|g\|_X$, we have
\begin{equation}
\circled{1} \ \lesssim_{p,q} \big\| \|f\|_X-\|g\|_X\big\|_{L_p(\mu;\R)}^{\min\{\frac{p}{q},1\}} \leq \big\|f-g\big\|_{L_p(\mu;X)}^{\min\{\frac{p}{q},1\}}.
\end{equation}
Moreover, if $\diff\nu = \|g\|_X^p \diff\mu$, then $\nu(\Omega)\leq 1$ and thus for $q\leq p$ we have
\begin{equation} \label{eq:bound-21}
\begin{split}
\circled{2} \ & = \|\theta-\phi\|_{L_q(\nu;X)} \leq \|\theta-\phi\|_{L_p(\nu;X)} = \big\| \|g\|_X \theta - \|g\|_X \phi \big\|_{L_p(\mu;X)}
\\ & \leq \big\| \|g\|_X - \|f\|_X  \big\|_{L_p(\mu;\R)} + \big\| \|f\|_X \theta - \|g\|_X \phi \big\|_{L_p(\mu;X)} \leq 2\|f-g\|_{L_p(\mu;X)}.
\end{split}
\end{equation}
On the other hand, if $p\leq q$, H\"older's inequality gives
\begin{equation}
\begin{split}
\circled{2} \ = \|\theta-\phi\|_{L_q(\nu;X)} \leq \|\theta-\phi\|_{L_p(\nu;X)}^{\frac{p}{q}}\cdot\|\theta-\phi\|_{L_\infty(\nu;X)}^{1-\frac{p}{q}} \stackrel{\eqref{eq:bound-21}}{\leq} 2 \|f-g\|_{L_p(\mu;X)}^{\frac{p}{q}},
\end{split}
\end{equation}
which completes the proof of the lemma.
\end{proof}

Equipped with the estimates of Lemma \ref{lem:mazur}, we now ready to prove Proposition \ref{prop:extrapolation}.

\medskip

\begin{proof}[Proof of Proposition \ref{prop:extrapolation}]
We shall first prove the rightmost inequality of \eqref{eq:extrapolation}. Let $x_1,\ldots,x_n\in X$ and consider the function $f:\{1,\ldots,n\}\to X$ given by $f(i)=x_i$. Unless all the vectors $x_1,\ldots,x_n$ are equal, we can rescale so that the constraint
\begin{equation} \label{eq:normalize0}
\sum_{i,j=1}^n \pi_i\pi_j \|x_i-x_j\|_X^p = 1
\end{equation}
is satisfied, which in particular, by Jensen's inequality, implies that \mbox{$\|f-\mb{E}_\pi f\|_{L_p(\pi;X)}\leq 1$,} where $\mb{E}_\pi f \eqdef \sum_{i=1}^n \pi_if(i)$. Consider the function $g=\msf{M}_{p,q}(f-\mb{E}_\pi f):\{1,\ldots,n\}\to X$. Then, we have
\begin{equation*}
\begin{split}
1&= \sum_{i,j=1}^n  \pi_i\pi_j  \|f(i)-f(j)\|_X^p = \sum_{i,j=1}^n \pi_i\pi_j \|\msf{M}_{q,p}g(i)-\msf{M}_{q,p}g(j)\|_X^p 
\\ & \stackrel{\eqref{eq:mazur}}{\lesssim}_{p,q} \Big(\sum_{i,j=1}^n \pi_i\pi_j \|g(i)-g(j)\|_X^q\Big)^{\frac{p}{q}} \leq \gamma(B,\|\cdot\|_X^q)^{\frac{p}{q}} \Big( \sum_{i,j=1}^n \pi_i b_{ij} \|g(i)-g(j)\|_X^q\Big)^{\frac{p}{q}}
\\ & = \gamma(B,\|\cdot\|_X^q)^{\frac{p}{q}} \Big( \sum_{i,j=1}^n \pi_i b_{ij} \|\msf{M}_{p,q}(f-\mb{E}_\pi f)(i)-\msf{M}_{p,q}(f-\mb{E}_\pi f)(j)\|_X^q\Big)^{\frac{p}{q}}
\\ & \stackrel{\eqref{eq:mazur}}{\lesssim}_{p,q} \! \gamma(B,\|\cdot\|_X^q)^{\frac{p}{q}} \Big(\sum_{i,j=1}^n \pi_i b_{ij} \|f(i)-f(j)\|_X^p \Big)^{\frac{p}{q}} \!\!\! = \gamma(B,\|\cdot\|_X^q)^{\frac{p}{q}} \Big(\sum_{i,j=1}^n \pi_i b_{ij} \|x_i-x_j\|_X^p \Big)^{\frac{p}{q}},
\end{split}
\end{equation*}
where in both inequalities we used that $\max\{\|f-\mb{E}_\pi f\|_{L_p(\pi;X)}, \|g\|_{L_q(\pi;X)}\}\leq 1$. Finally, taking an infimum of the right-hand side over all $x_1,\ldots,x_n\in X$ satisfying \eqref{eq:normalize0}, we deduce that $1\lesssim_{p,q} \gamma(B,\|\cdot\|_X^q)^{p/q} \gamma(B,\|\cdot\|_X^p)^{-p/q}$ which concludes the proof.

The proof of the leftmost inequality is almost identical yet we repeat it for completeness. Let $x_1,\ldots,x_n\in X$ and consider the function $\phi:\{1,\ldots,n\}\to X$ given by $\phi(i)=x_i$. Without loss of generality, we can again assume that the constraint
\begin{equation} \label{eq:normalize}
\sum_{i,j=1}^n \pi_i\pi_j \|x_i-x_j\|_X^q = 1, 
\end{equation}
is satisfied, which implies that $\|\phi-\mb{E}_\pi\phi\|_{L_q(\pi;X)}\leq 1$ by Jensen's inequality. Consider the function $\psi=\msf{M}_{q,p}(\phi-\mb{E}_\pi\phi):\{1,\ldots,n\}\to X$. Then, we have
\begin{equation*}
\begin{split}
1 &= \sum_{i,j=1}^n \pi_i\pi_j \|\phi(i)-\phi(j)\|_X^q = \sum_{i,j=1}^n \pi_i\pi_j \|\msf{M}_{p,q}\psi (i) - \msf{M}_{p,q}\psi(j)\|_X^q
\\ & \stackrel{\eqref{eq:mazur}}{\lesssim}_{p,q} \sum_{i,j=1}^n \pi_i\pi_j \|\psi(i)-\psi(j)\|_X^p \leq \gamma(B,\|\cdot\|_X^p) \sum_{i,j=1}^n\pi_i b_{ij} \|\psi(i)-\psi(j)\|_X^p 
\\ & = \gamma(B,\|\cdot\|_X^p) \sum_{i,j=1}^n \pi_i b_{ij} \|\msf{M}_{q,p}(\phi-\mb{E}_\pi\phi)(i)-\msf{M}_{q,p}(\phi-\mb{E}_\pi\phi)(j)\|_X^p 
\\ & \stackrel{\eqref{eq:mazur}}{\lesssim}_{p,q} \gamma(B,\|\cdot\|_X^p) \Big(\sum_{i,j=1}^n \pi_i b_{ij} \|\phi(i)-\phi(j)\|_X^q\Big)^{\frac{p}{q}} = \gamma(B,\|\cdot\|_X^p) \Big(\sum_{i,j=1}^n \pi_i b_{ij} \|x_i-x_j\|_X^q\Big)^{\frac{p}{q}}
\end{split}
\end{equation*}
where in both inequalities we used that $\max\{\|\phi-\mb{E}_\pi\phi\|_{L_q(\pi;X)},\|\psi\|_{L_p(\pi;X)}\} \leq1$. Finally, taking an infimum of the right-hand side over all $x_1,\ldots,x_n\in X$ satisfying \eqref{eq:normalize}, we deduce that $1\lesssim_{p,q} \gamma(B,\|\cdot\|_X^p) \gamma(B,\|\cdot\|_X^q)^{-p/q}$ which concludes the proof.
\end{proof}

Even though the vector-valued version of Matou\v{s}ek's extrapolation theorem (Proposition \ref{prop:extrapolation}) suffices for the proof of Theorem \ref{thm:john} which will be presented in the next section, we digress to mention the following stronger result of \textcite{Nao21}.

\begin{prop} \label{prop:snowflake}
Fix $p\in(0,\infty)$ and $\theta\in(0,1)$. There exists \mbox{$D_0=D_0(p,\theta)\in(1,\infty)$} such that the $\theta$-snowflake of any normed space $(X,\|\cdot\|_X)$ embeds with $p$-average distortion $D_0$ into $X$.
\end{prop}

The discussion preceeding the statement of Theorem \ref{thm:duality} shows that Proposition \ref{prop:extrapolation} is a formal consequence of Proposition \ref{prop:snowflake}, whose proof also relies on properties of the vector-valued Mazur map \eqref{eq:mazur-map}. In fact, as explained in \textcite[Remark~47]{Nao21}, Proposition \ref{prop:snowflake} implies improved bounds for the parameters $c(p,q), C(p,q)$ appearing in Proposition \ref{prop:extrapolation}. Proposition \ref{prop:snowflake} is the average distortion analogue of the following classical open problem in the nonlinear theory of Banach spaces.

\begin{question} \label{q:snowflake}
Does there exist $\theta\in(0,1)$ and an infinite-dimensional Banach space $(X,\|\cdot\|_X)$ whose $\theta$-snowflake does not admit a bi-Lipschitz embedding into $X$?
\end{question}

A classical result of \textcite{Sch38} implies that for every $\theta\in(0,1)$, the $\theta$-snowflake of $L_2$ admits an isometric embedding into $L_2$. Schoenberg's theorem was later extended by \textcite{BDCK66}, who showed that for every $p\in(0,2]$ and $\theta\in(0,1)$, the $\theta$-snowflake of $L_p$ admits an isometric embedding into $L_p$. Despite decades of attention, Question \ref{q:snowflake} remains stubbornly open even for the spaces $X=L_p$, where $p\in(2,\infty)$. In the forthcoming work of \textcite{EN21}, it is proven that arbitrarily small logarithmic perturbations of this question have a negative answer. More precisely, it is shown that for every $\theta\in(0,1)$, $\eta\in(0,\min\{\theta,1-\theta\})$ and $p\in(2,\infty)$, the metric transforms $(L_p,\omega_{\theta,\eta}\circ d_{L_p})$ of $L_p$ do not admit a bi-Lipschitz embedding into $L_p$, where 
\begin{equation}
\omega_{\theta,\eta}(t) = t^\theta \log^\eta (1+t) \qquad \mbox{or} \qquad\omega_{\theta,\eta}(t) = \frac{t^\theta}{1+\log^\eta(1+t)}.
\end{equation}
We refer to \textcite[Chapter~3]{Esk19} for further results in this direction.


\section{Proof of the average John theorem} \label{sec:estimate}

Having established the duality principle of Theorem \ref{thm:duality} and the extrapolation inequalities of Proposition \ref{prop:extrapolation}, we are well equipped to proceed to the proof of Theorem \ref{thm:john} via Theorem \ref{thm:estimate}. We start with some preliminary properties of \emph{nonlinear Rayleigh quotients} which will help us analyze nonlinear spectral gaps. The following simplification of the original proof \mbox{of Theorem \ref{thm:john} was sketched in \textcite[Remark~31]{Nao21}.}


\subsection{Nonlinear Rayleigh quotients} 
Fix $p\geq1$, a metric space $(\MM,d_\MM)$ and a probability measure $\pi\in\triangle^{n-1}$. Let $L_p(\pi;\MM)$ be the metric space $(\MM^n,d_{L_p(\pi,\MM)})$ whose metric is given by 
\begin{equation} 
d_{L_p(\pi;\MM)}(\x,\y)= \Big(\sum_{i=1}^n \pi_i d_\MM(x_i,y_i)^p\Big)^{\frac{1}{p}},
\end{equation}
where $\x=(x_1,\ldots,x_n), \y=(y_1,\ldots,y_n)\in\MM^n$. Moreover, we shall use the ad hoc notation $L_p(\pi;\MM)^\dagger$ for the subset $L_p(\pi;M)\setminus\{(x,\ldots,x): \ x\in\MM\}$ of $L_p(\pi;\MM)$.

Let $A\in M_n(\R)$ be a row-stochastic matrix that is $\pi$-stationary (in the sense that $\pi A=\pi$) and $\x=(x_1,\ldots,x_n)\in L_p(\pi;\MM)^\dagger$. Following \textcite[Section~5.1]{Nao18}, we consider the corresponding nonlinear Rayleigh quotient given by
\begin{equation} \label{eq:nrq}
\mc{R}(\x;A,d_\MM^p) \eqdef \frac{\sum_{i,j=1}^n \pi_i a_{ij} d_\MM(x_i,x_j)^p}{\sum_{i,j=1}^n \pi_i\pi_j d_\MM(x_i,x_j)^p}.
\end{equation}
By definition, if $A$ is $\pi$-reversible, the nonlinear spectral gap \eqref{eq:nsg} satisfies
\begin{equation} \label{eq:nrq-sg}
\gamma(A,d_\MM^p) = \sup_{\x\in L_p(\pi;\MM)^\dagger} \frac{1}{\mc{R}(\x;A,d_\MM^p)}.
\end{equation}
We will need the following properties of nonlinear Rayleigh quotients.

\begin{lemm} \label{lem:nrq}
Let $(\MM,d_\MM)$ be a metric space, $p\geq1$, $\lambda\in[0,1]$ and $\pi\in\triangle^{n-1}$. If $A,B\in M_n(\R)$ are $\pi$-stationary stochastic matrices and $\x\in L_p(\pi;\MM)^\dagger$, then we have
\begin{enumerate}
\item[(i)] $\mc{R}(\x;\lambda A+(1-\lambda)B,d_\MM^p) = \lambda \mc{R}(\x;A,d_\MM^p) + (1-\lambda)\mc{R}(\x;B,d_\MM^p)$.
\item[(ii)] $\mc{R}(\x;\lambda A+(1-\lambda)\msf{Id}_n,d_\MM^p) = \lambda \mc{R}(\x;A,d_\MM^p)$, where $\msf{Id}_n$ is the identity matrix.
\item[(iii)] $\mc{R}(\x;AB,d_\MM^p)^{\frac{1}{p}} \leq \mc{R}(\x;A,d_\MM^p)^{\frac{1}{p}}+\mc{R}(\x;B,d_\MM^p)^{\frac{1}{p}}$.
\item[(iv)] $\mc{R}(\x;B^t,d_\MM^p) \leq t^p \mc{R}(\x;B,d_\MM^p)$ for every $t\in\N$.
\end{enumerate}
\end{lemm}

\begin{proof}
The first property is evident from the definition \eqref{eq:nrq} and the second follows from (i) since $\mc{R}(\x;\msf{Id}_n,d_\MM^p)=0$. Moreover, (iv) follows by iterating (iii) so we are left to prove that. Notice that $AB$ is $\pi$-stationary and the triangle inequality gives
\begin{equation} \label{eq:iii}
\begin{split}
\Big(\sum_{i,j=1}^n& \pi_i (AB)_{ij} d_\MM(x_i,x_j)^p\Big)^{\frac{1}{p}} \leq \Big(\sum_{i,j=1}^n \pi_i \sum_{k=1}^n a_{ik}b_{kj} \big(d_\MM(x_i,x_k)+d_\MM(x_k,x_j)\big)^p\Big)^{\frac{1}{p}}
\\ & \leq \Big( \sum_{i,j,k=1}^n \pi_ia_{ik}b_{kj} d_\MM(x_i,x_k)^p \Big)^{\frac{1}{p}} + \Big( \sum_{i,j,k=1}^n \pi_i a_{ik}b_{kj} d_\MM(x_k,x_j)^p \Big)^{\frac{1}{p}}
\\ & = \Big( \sum_{i,k=1}^n \pi_i a_{ik} d_\MM(x_i,x_k)^p\Big)^{\frac{1}{p}} + \Big( \sum_{k,j=1}^n \pi_k b_{kj} d_\MM(x_k,x_j)^p\Big)^{\frac{1}{p}},
\end{split}
\end{equation}
where in the last equality we used the stationarity of $A$ in the form $\pi_k = \sum_{i=1}^n \pi_i a_{ik}$. The desired inequality (iii) follows from \eqref{eq:iii} after renormalizing.
\end{proof}

Apart from the elementary properties of Lemma \ref{lem:nrq}, we shall also need the following standard computation of nonlinear Rayleigh quotients in Hilbert space. Recall that for every normed space $(X,\|\cdot\|_X)$, every matrix $B\in M_n(\R)$ induces a linear operator $B\otimes\msf{Id}_X:L_p(\pi;X)\to L_p(\pi;X)$ that is given by $(B\otimes\msf{Id}_X)(x_1,\ldots,x_n) = (\sum_{j=1}^n b_{ij} x_j)_{i=1}^n$.

\begin{lemm} \label{lem:nrq-hilbert}
Fix $\pi\in\triangle^{n-1}$ and let $B\in M_n(\R)$ be a $\pi$-reversible stochastic matrix. For every Hilbert space $(\mc{H},\|\cdot\|_{\mc{H}})$ and $\x\in L_2(\pi;\mc{H})^\dagger$ with $\sum_{i=1}^n\pi_i x_i=0$, we have
\begin{equation}
\mc{R}(\boldsymbol{x};B^2,\|\cdot\|_{\mc{H}}^2) = 1-\frac{\|(B\otimes\msf{Id}_\mc{H})\x\|_{L_2(\pi;\mc{H})}^2}{\|\x\|_{L_2(\pi;\mc{H})}^2}.
\end{equation}
\end{lemm}

\begin{proof}
Let $\langle\cdot,\cdot\rangle_\mc{H}$ be the inner product of $\mc{H}$ and notice that
\begin{equation}
\sum_{i,j=1}^n \pi_i\pi_j \|x_i-x_j\|_\mc{H}^2 = 2 \sum_{i=1}^n\pi_i \|x_i\|_\mc{H}^2 - 2\left\langle \sum_{i=1}^n\pi_i x_i, \sum_{j=1}^n \pi_jx_j\right\rangle_\mc{H} = 2\|\x\|_{L_2(\pi;\mc{H})}^2,
\end{equation}
since $\sum_{i=1}^n\pi_i x_i=0$. Moreover, since $B^2$ is $\pi$-stationary and stochastic, we have
\begin{equation*}
\begin{split}
&\sum_{i,j=1}^n \pi_i (B^2)_{ij} \|x_i-x_j\|_\mc{H}^2 = \sum_{i,j=1}^n \pi_i (B^2)_{ij} (\|x_i\|_\mc{H}^2+\|x_j\|_\mc{H}^2)-2\sum_{i=1}^n \pi_i \left\langle x_i, \sum_{j=1}^n (B^2)_{ij} x_j\right\rangle_\mc{H}
\\ & = 2 \|\x\|_{L_2(\pi;\mc{H})}^2 - 2 \sum_{i=1}^n \pi_i \big\langle x_i, \big( (B^2\otimes\msf{Id}_\mc{H})\x\big)_i\big\rangle_\mc{H} = 2 \|\x\|_{L_2(\pi;\mc{H})}^2 - 2\|(B\otimes\msf{Id}_\mc{H})\x\|_{L_2(\pi;\mc{H})}^2,
\end{split}
\end{equation*}
where in the last equality we additionally used the $\pi$-reversibility of $B$. The conclusion now readily follows by the definition \eqref{eq:nrq} of nonlinear Rayleigh quotients.
\end{proof}


\subsection{Proof of Theorem \ref{thm:estimate}}

In the proof of Theorem \ref{thm:estimate} we will use the following pointwise estimate of nonlinear Rayleigh quotients of normed spaces which are isomorphic to a Hilbert space.

\begin{lemm} \label{lem:pointwise}
Let $(X,\|\cdot\|_X)$ be a normed space and $D\in[1,\infty)$. Suppose that there exists a Hilbertian norm $\|\cdot\|_\mc{H}:X\to\R_+$ such that
\begin{equation} \label{eq:isomorph}
\forall \ y\in X, \qquad \|y\|_\mc{H}\leq \|y\|_X \leq D \|y\|_\mc{H}.
\end{equation}
Then, for every $\pi\in\triangle^{n-1}$ and every $\pi$-reversible stochastic matrix $B\in M_n(\R)$,
\begin{equation} \label{eq:iso-impli}
\mc{R}(\x;B^2,\|\cdot\|_\mc{H}^2) \geq 1-\eta^2 \quad \Longrightarrow \quad \mc{R}(\x;B,\|\cdot\|_X^2) \geq \frac{(1-\eta D)^2}{4},
\end{equation}
where $\x\in L_2(\pi;X)^\dagger$ and $\eta\in(0,1/D)$.
\end{lemm}

\begin{proof}
Without loss of generality, we can translate the components $x_i$ of the vector $\x\in L_2(\pi;X)^\dagger$ to assume that $\sum_{i=1}^n\pi_i x_i=0$. Then, the assumption $\mc{R}(\x;B^2,\|\cdot\|_\mc{H}^2)\geq 1-\eta^2$ can  be equivalently rewritten due to Lemma \ref{lem:nrq-hilbert} as
\begin{equation} \label{eq:uselema}
\|(B\otimes\msf{Id}_\mc{H})\x\|_{L_2(\pi;\mc{H})} \leq \eta \|\x\|_{L_2(\pi;\mc{H})}.
\end{equation}
Therefore,
\begin{equation} \label{useisom}
\|(B\otimes\msf{Id}_X)\x\|_{L_2(\pi;X)} \stackrel{\eqref{eq:isomorph}}{\leq} D\|(B\otimes\msf{Id}_\mc{H})\x\|_{L_2(\pi;\mc{H})}  \stackrel{\eqref{eq:uselema}}{\leq} \eta D \|\x\|_{L_2(\pi;\mc{H})} \stackrel{\eqref{eq:isomorph}}{\leq} \eta D\|\x\|_{L_2(\pi;X)},
\end{equation}
and thus, by the triangle inequality,
\begin{equation} \label{eq:11}
\|\x-(B\otimes \msf{Id}_X)\x\|_{L_2(\pi;X)} \geq \|\x\|_{L_2(\pi;X)} - \|(B\otimes\msf{Id}_X)\x\|_{L_2(\pi;X)} \stackrel{\eqref{useisom}}{\geq} (1-\eta D) \|\x\|_{L_2(\pi;X)}.
\end{equation}
Since $B$ is row-stochastic, Jensen's inequality for the convex function $\|\cdot\|_X^2$ gives
\begin{equation} \label{eq:22}
\sum_{i,j=1}^n \pi_i b_{ij} \|x_i-x_j\|_X^2 \geq \sum_{i=1}^n \pi_i \Big\|x_i - \sum_{j=1}^n b_{ij} x_j\Big\|_X^2 = \|\x-(B\otimes \msf{Id}_X)\x\|_{L_2(\pi;X)}^2.
\end{equation}
On the other hand, using the triangle inequality we get
\begin{equation} \label{eq:33}
\sum_{i,j=1}^n \pi_i\pi_j \|x_i-x_j\|_X^2 \leq \sum_{i,j=1}^n \pi_i\pi_j (\|x_i\|_X+\|x_j\|_X)^2 \leq 4 \|\x\|_{L_2(\pi;X)}^2.
\end{equation}
Combining \eqref{eq:nrq}, \eqref{eq:22}, \eqref{eq:11} and \eqref{eq:33} we deduce that
\begin{equation}
\mc{R}(\x;B,\|\cdot\|_X^2) \geq \frac{(1-\eta D)^2}{4},
\end{equation}
which concludes the proof.
\end{proof}

Equipped with Lemma \ref{lem:pointwise}, we can complete the proof of Theorem \ref{thm:estimate}. The main idea is to consider a Hilbertian norm which nicely approximates our given norm on $\R^d$ and then use the implication \eqref{eq:iso-impli}. In order to ensure that the assumption of \eqref{eq:iso-impli} is satisfied we shall apply a trick that was used by \textcite{Pis10}, who attributed it to V.~Lafforgue: we will replace $A$ by a large enough power of the form $\big(\tfrac{A+\msf{Id}_n}{2}\big)^t$. We will then be able to return to an inequality involving \mbox{$A$ rather than its power using Lemma \ref{lem:nrq}.}

\medskip

\begin{proof} [Proof of Theorem \ref{thm:estimate}]
Suppose that $X=(\R^d,\|\cdot\|_X)$ and fix $\pi\in\triangle^{n-1}$, a $\pi$-reversible stochastic matrix $A\in M_n(\R)$ and a vector $\x\in L_1(\pi;X)^\dagger$. In view of \eqref{eq:nrq-sg}, we need to prove a lower bound on $\mc{R}(\x;A,\|\cdot\|_X)$. Notice that, by properties (ii) and (iv) of Lemma \ref{lem:nrq}, we have the inequality
\begin{equation} \label{eq:take-power}
\mc{R}(\x;A,\|\cdot\|_X) = 2 \mc{R}\Big(\x;\frac{A+\msf{Id}_n}{2},\|\cdot\|_X\Big) \geq \frac{2}{t}\mc{R}\Big(\x;\Big(\frac{A+\msf{Id}_n}{2}\Big)^t,\|\cdot\|_X\Big)
\end{equation}
for every $t\in\N$. Moreover, by the vector-valued extrapolation inequalities of Proposition \ref{prop:extrapolation} and the expression \eqref{eq:nrq-sg} of nonlinear spectral gaps in terms of nonlinear Rayleigh quotients, we conclude that there exists a point $\y\in L_2(\pi;X)^\dagger$ satisfying
\begin{equation} \label{eq:use-extra}
\mc{R}\Big(\x;\Big(\frac{A+\msf{Id}_n}{2}\Big)^t,\|\cdot\|_X\Big) \gtrsim \mc{R}\Big(\y;\Big(\frac{A+\msf{Id}_n}{2}\Big)^t,\|\cdot\|_X^2\Big).
\end{equation}

Let $D_X\in[1,\infty)$ be the least constant for which there exists a Hilbertian norm $\|\cdot\|_\mc{H}:\R^d\to\R_+$ such that the following inequality is satisfied,
\begin{equation} \label{eq:hilbe-norm}
\forall \ y\in \R^d, \qquad \|y\|_\mc{H}\leq \|y\|_X \leq D_X \|y\|_\mc{H}.
\end{equation}
As $\mc{H}$ is isometric to $\ell_2^d$, the spectral gap of $\big(\tfrac{A+\msf{Id}_n}{2}\big)^{2t}$ with respect to $\|\cdot\|^2_\mc{H}$ satisfies
\begin{equation} \label{eq:eucl-gamma}
\gamma\Big( \Big(\frac{A+\msf{Id}_n}{2}\Big)^{2t},\|\cdot\|_\mc{H}^2\Big) = \gamma\Big(\Big(\frac{A+\msf{Id}_n}{2}\Big)^{2t},|\cdot|^2\Big) = \frac{1}{1-\big(\frac{1+\lambda_2(A)}{2}\big)^{2t}}.
\end{equation}
Therefore, for the parameter
\begin{equation} \label{eq:power}
t^\ast(A)\eqdef \left\lceil\frac{\log (2D_X)}{\log(\frac{2}{1+\lambda_2(A)})}\right\rceil \lesssim \frac{\log(D_X+1)}{1-\lambda_2(A)}
\end{equation}
we have the estimate
\begin{equation}
\gamma\Big( \Big(\frac{A+\msf{Id}_n}{2}\Big)^{2t^*(A)},\|\cdot\|_\mc{H}^2\Big) \stackrel{\eqref{eq:eucl-gamma}\wedge\eqref{eq:power}}{\leq} \frac{1}{1-\frac{1}{4D_X^2}},
\end{equation}
which combined with \eqref{eq:nrq-sg} immediately implies that
\begin{equation} \label{eq:got-hilbert-bound}
\mc{R}\Big(\y;\Big(\frac{A+\msf{Id}_n}{2}\Big)^{2t^*(A)},\|\cdot\|_\mc{H}^2\Big) \geq 1-\frac{1}{4D_X^2}.
\end{equation}
Therefore, in view of \eqref{eq:hilbe-norm} and \eqref{eq:got-hilbert-bound}, the pointwise estimate of Lemma \ref{lem:pointwise} applied to the matrix $B=\big(\tfrac{A+\msf{Id}_n}{2}\big)^{t^*(A)}$ and $\eta=\tfrac{1}{2D_X}$ implies that
\begin{equation} \label{eq:got-X-bound}
\mc{R}\Big(\y;\Big(\frac{A+\msf{Id}_n}{2}\Big)^{t^*(A)},\|\cdot\|_X^2\Big) \geq \frac{1}{16}.
\end{equation}
Finally, combining \eqref{eq:take-power}, \eqref{eq:use-extra} and \eqref{eq:got-X-bound} for $t=t^\ast(A)$ we deduce that
\begin{equation} \label{eq:almost-done}
\mc{R}(\x;A,\|\cdot\|_X) \gtrsim \frac{1}{t^*(A)} \stackrel{\eqref{eq:power}}{\gtrsim} \frac{1-\lambda_2(A)}{\log(D_X+1)}.
\end{equation}
By John's theorem \parencite{Joh48}, since $X$ is $d$-dimensional we have $D_X\leq \sqrt{d}$ and the  desired estimate \eqref{eq:estimate} thus follows by rearranging \eqref{eq:almost-done} and using \eqref{eq:nrq-sg}.
\end{proof}

As explained in the introduction, Theorem \ref{thm:john} is equivalent to  Theorem \ref{thm:estimate} via the duality principle of Theorem \ref{thm:duality}.

\medskip

\begin{proof} [Proof of Theorem \ref{thm:john}]
Combining Theorems \ref{thm:estimate} and \ref{thm:duality}, we deduce that the $\tfrac{1}{2}$-snowflake of any $d$-dimensional normed space embeds into an ultrapower of $\ell_2$ with quadratic average distortion at most $C\sqrt{\log d}$, where $d\geq2$ and $C\in(0,\infty)$ is a universal constant. This immediately yields the conclusion of Theorem \ref{thm:john} since any ultrapower of $\ell_2$ is itself a Hilbert space \parencite{Hei80}.
\end{proof}


\section{Beyond Hilbertian embeddings} \label{sec:beyond}

Theorem \ref{thm:john} is a special case of a much more general embedding theorem proven by \textcite{Nao21}. As a matter of fact, a lot of the ideas required to prove this more general statement have already been used in the Hilbertian case. A key ingredient required to go beyond Theorem \ref{thm:john} is the notion of \emph{Markov type} introduced by \textcite{Bal92}.

\begin{defi}
A metric space $(\MM,d_\MM)$ has Markov type $p\in(0,\infty)$ with constant $M\in(0,\infty)$ if for every $n\in\N$, $\pi\in\triangle^{n-1}$, every $\pi$-reversible matrix $A\in M_n(\R)$ and every $\x\in L_p(\pi;\MM)$, we have
\begin{equation} \label{eq:mtype}
\forall \ t\in\N, \qquad \mc{R}(\x;A^t,d_\MM^p) \leq M^pt\ \mc{R}(\x;A,d_\MM^p).
\end{equation}
The least such constant $M\in(0,\infty)$ will be denoted by $\msf{M}_p(\MM)$.
\end{defi}

In \textcite{Bal92}, it was shown that any Hilbert space $\mc{H}$ has $\msf{M}_2(\mc{H})=1$. Following \textcite{BCL94}, we say that a normed space $(X,\|\cdot\|_X)$ is $p$-uniformly smooth, where $p\in[1,2]$, if there exists a constant $S\in(0,\infty)$ such that
\begin{equation}
\forall \ x,y\in X, \qquad \frac{\|x\|_X^p+\|y\|_X^p}{2} \leq \Big\| \frac{x+y}{2}\Big\|_X^p + S^p \Big\| \frac{x-y}{2} \Big\|_X^p.
\end{equation}
The least such constant $S\in(0,\infty)$ will be denoted by $\msf{S}_p(X)$. A deep theorem of \textcite{NPSS06} asserts that every $p$-uniformly smooth normed space $(X,\|\cdot\|_X)$ has Markov type $p$ with constant
\begin{equation} \label{eq:npss}
\msf{M}_p(X) \lesssim \msf{S}_p(X).
\end{equation}

An inspection of the proof of Theorem \ref{thm:estimate} reveals that the power of the norm in the estimate \eqref{eq:estimate} can be improved for spaces of Markov type $p\in(1,2]$. Indeed, replacing \eqref{eq:take-power} with \eqref{eq:mtype} and using the extrapolation inequality \eqref{eq:extrapolation}, we deduce that for every $\x\in L_p(\pi;X)^\dagger$ there exists $\y\in L_2(\pi;X)^\dagger$ for which we have the estimate
\begin{equation} \label{eq:use-mtype}
\begin{split}
\mc{R}(\x;A,&\|\cdot\|_X^p) = 2\mc{R}\Big(\x;\frac{A+\msf{Id}_n}{2},\|\cdot\|_X^p\Big)  \\ & \stackrel{\eqref{eq:mtype}}{\gtrsim}_{_{\msf{M}_p(X)}} \frac{1}{t} \mc{R}\Big(\x;\Big(\frac{A+\msf{Id}_n}{2}\Big)^t,\|\cdot\|_X^p\Big)
 \stackrel{\eqref{eq:extrapolation}}{\gtrsim} \frac{1}{t} \mc{R}\Big(\y;\Big(\frac{A+\msf{Id}_n}{2}\Big)^t,\|\cdot\|_X^2\Big).
\end{split}
\end{equation}
Then, repeating the rest of the proof mutatis mutandis, we deduce the bound
\begin{equation}
\gamma(A,\|\cdot\|_X^p) \lesssim_{_{\msf{M}_p(X)}} \frac{\log(\mathrm{dim}(X)+1)}{1-\lambda_2(A)},
\end{equation}
which, in view of \eqref{eq:npss} and Theorem \ref{thm:duality}, implies the following embeddability result.

\begin{theo} \label{thm:p/2}
For every $S\in(0,\infty)$, there exists $C(S)\in(0,\infty)$ such that the following holds. If $p\in[1,2]$ and $(X,\|\cdot\|_X)$ is a finite-dimensional normed space with $\msf{S}_p(X)\leq S$, then the $\tfrac{p}{2}$-snowflake of $(X,\|\cdot\|_X)$ admits an embedding into $\ell_2$ with quadratic average distortion at most $C(S)\sqrt{\log (\mathrm{dim}(X)+1)}$.
\end{theo}

Theorem \ref{thm:john} is a special case of Theorem \ref{thm:p/2} as $\msf{S}_1(X)=1$ for any $(X,\|\cdot\|_X)$. However, Theorem \ref{thm:p/2} is a refinement of the average John theorem in that it captures the fact that more structured normed spaces (i.e.~spaces with bounded $p$-uniform smoothness constant) require a lesser amount of snowflaking in order to be embedded into $\ell_2$ with quadratic average distortion which depends subpolynomially on the dimension. It is worth emphasizing that for 2-uniformly smooth spaces (such as $L_r(\mu)$ with $2< r<\infty$), Theorem \ref{thm:p/2} shows that \emph{no} snowflaking is necessary for such an embedding to exist.

This approach can be further exploited even for target spaces which are not Hilbertian. Following \textcite{BCL94}, we say that a normed space $(X,\|\cdot\|_X)$ is $q$-uniformly convex, where $q\in[2,\infty)$, if there exists a constant $K\in(0,\infty)$ such that
\begin{equation}
\forall \ x,y\in X,\qquad \Big\|\frac{x+y}{2}\Big\|_X^q + \frac{1}{K^q}\Big\|\frac{x-y}{2}\Big\|_X^q \leq \frac{\|x\|_X^q+\|y\|_X^q}{2}.
\end{equation}
The least such constant $K\in(0,\infty)$ will be denoted by $\msf{K}_q(X)$. Observe that $\msf{K}_2(\ell_2)=1$.   Theorem \ref{thm:p/2} admits the following non-Hilbertian generalization. We shall denote by $\msf{c}_Y(\MM)$ the infimal distortion of a bi-Lipschitz embedding $f:(\MM,d_\MM)\to(Y,\|\cdot\|_Y)$.

\begin{theo} \label{thm:p/q}
For every $S,K\in(0,\infty)$, there exists $C(S,K)\in(0,\infty)$ such that the following holds. If $1\leq p\leq 2\leq q <\infty$, $(X,\|\cdot\|_X)$ is a Banach space with $\msf{S}_p(X)\leq S$ and $(Y,\|\cdot\|_Y)$ is a Banach space with $\msf{K}_q(Y)\leq K$, then the $\frac{p}{q}$-snowflake of $X$ admits an embedding into $\ell_q(Y)$ with $q$-average distortion at most $C(S,K)(\log(\msf{c}_Y(X)+1))^{1/q}$.
\end{theo}

In view of the duality principle\footnote{Observe that a direct application of Theorem \ref{thm:duality} and \eqref{eq:p/q} would imply that the $\tfrac{p}{q}$-snowflake of $X$ admits such an embedding into an ultrapower of $\ell_q(Y)$ rather than $\ell_q(Y)$ itself. Taking an ultrapower in this statement is redundant if $X$ is assumed to be $p$-uniformly smooth and $Y$ is $q$-uniformly convex, as was shown in \textcite[Corollary~23]{Nao21}. We shall not address this delicate issue here.} of Theorem \ref{thm:duality}, Theorem \ref{thm:p/q} is equivalent the following nonlinear spectral gap inequality. For every $n\in\N$, $\pi\in\triangle^{n-1}$ and every $\pi$-reversible matrix $A\in M_n(\R)$, we have
\begin{equation} \label{eq:p/q}
\gamma(A,\|\cdot\|_X^p) \lesssim_{_{\msf{M}_p(X), \msf{K}_q(Y)},p,q} \log(\msf{c}_Y(X)+1) \gamma(A,\|\cdot\|_Y^q).
\end{equation}
Fix $\x\in L_p(\pi;X)^\dagger$ with $\sum_{i=1}^n\pi_i x_i=0$. Using Markov type and extrapolation as in \eqref{eq:use-mtype}, we deduce that for any $t\in\N$, there exists $\y\in L_q(\pi;X)^\dagger$ with $\sum_{i=1}^n \pi_i y_i=0$ such that
\begin{equation} \label{eq:222}
\mc{R}(\x;A,\|\cdot\|_X^p) \gtrsim_{_{\msf{M}_p(X)},p,q} \frac{1}{t} \mc{R}\Big(\y; \Big(\frac{A+\msf{Id}_n}{2}\Big)^t, \|\cdot\|_X^q\Big).
\end{equation}
Moreover, if $B_t\eqdef(\tfrac{A+\msf{Id}_n}{2})^t$, the argument of \eqref{eq:22} and \eqref{eq:33} implies that
\begin{equation} \label{eq:333}
\begin{split}
\mc{R}(\y; B_t, \|\cdot\|_X^q) \gtrsim_q & \frac{\|\y - (B_t\otimes\msf{Id}_X)\y\|_{L_q(\pi;X)}^q}{\|\y\|^q_{L_q(\pi;X)}} \\ & \geq \big(1- \|B_t\otimes\msf{Id}_X\|_{L_q^{^0}(\pi;X)\to L_q^{^0}(\pi;X)} \big)^q,
\end{split}
\end{equation}
where $L_q^{^0}(\pi;Z) = \{\boldsymbol{z}\in L_q(\pi;Z): \ \sum_{i=1}^n \pi_i z_i=0\}$. Therefore, we have
\begin{equation} \label{eq:41}
\gamma(A,\|\cdot\|_X^p) \stackrel{\eqref{eq:222}\wedge\eqref{eq:nrq-sg}\wedge\eqref{eq:333}}{\lesssim_{_{\msf{M}_p(X)},p,q}} \frac{t}{(1- \|B_t\otimes\msf{Id}_X\|_{L_q^{^0}(\pi;X)\to L_q^{^0}(\pi;X)})^q}.
\end{equation}
Notice that by the definition of $\msf{c}_Y(X)$,
\begin{equation} \label{eq:42}
\|B_t\otimes\msf{Id}_X\|_{L_q^{^0}(\pi;X)\to L_q^{^0}(\pi;X)} \leq \msf{c}_Y(X) \|B_t\otimes\msf{Id}_Y\|_{L_q^{^0}(\pi;Y)\to L_q^{^0}(\pi;Y)} 
\end{equation}
and moreover
\begin{equation} \label{eq:43}
\begin{split}
\|B_t\otimes\msf{Id}_Y\|_{L_q^{^0}(\pi;Y)\to L_q^{^0}(\pi;Y)}  = \Big\|\Big(\frac{A+\msf{Id}_n}{2}\Big)^t&\otimes\msf{Id}_Y\Big\|_{L_q^{^0}(\pi;Y)\to L_q^{^0}(\pi;Y)} 
\\ & \leq \Big\|\Big(\frac{A+\msf{Id}_n}{2}\Big)\otimes\msf{Id}_Y\Big\|_{L_q^{^0}(\pi;Y)\to L_q^{^0}(\pi;Y)}^t.
\end{split}
\end{equation}
Combining \eqref{eq:41}, \eqref{eq:42} and \eqref{eq:43}, we finally deduce that for any $t\in\N$,
\begin{equation} 
\gamma(A,\|\cdot\|_X^p)\lesssim_{_{\msf{M}_p(X)},p,q} t \cdot \Big( 1- \msf{c}_Y(X) \Big\|\Big(\frac{A+\msf{Id}_n}{2}\Big)\otimes\msf{Id}_Y\Big\|_{L_q^{^0}(\pi;Y)\to L_q^{^0}(\pi;Y)}^t \Big)^{-q}.
\end{equation}
Optimizing over $t$ we thus conclude that
\begin{equation} \label{eq:66}
\gamma(A,\|\cdot\|_X^p) \lesssim_{_{\msf{M}_p(X)},p,q} \frac{\log(\msf{c}_Y(X)+1)}{\log\big(1/\|(\frac{A+\msf{Id}_n}{2})\otimes\msf{Id}_Y\|_{L_q^{^0}(\pi;Y)\to L_q^{^0}(\pi;Y)}\big)}
\end{equation}
Observe that so far we have been very closely following the Hilbertian proof. Indeed, if $Y$ is a Hilbert space and $q=2$, then the operator norm appearing in \eqref{eq:66} is simply $\tfrac{1+\lambda_2(A)}{2}$ and thus \eqref{eq:estimate} follows from \eqref{eq:66} and John's theorem which asserts that $\msf{c}_{\ell_2}(X)\leq \sqrt{\mathrm{dim}X}$. In the general (Banach space-valued) setting of Theorem \ref{thm:p/q}, we need a more robust argument to show that the operator norm $\big\|(\tfrac{A+\msf{Id}_n}{2})\otimes\msf{Id}_Y\big\|_{L_q^{^0}(\pi;Y)\to L_q^{^0}(\pi;Y)}$ is bounded away from 1 by a quantity which depends on the nonlinear spectral gap $\gamma(A,\|\cdot\|_Y^q)$. To do this, we will leverage the $q$-uniform convexity of the normed space $(Y,\|\cdot\|_Y)$.

Fix a metric space $(\MM,d_\MM)$ and $q\in(0,\infty)$. If $\pi\in\triangle^{n-1}$ and $A\in M_n(\R)$ is a $\pi$-reversible stochastic matrix, the \emph{nonlinear absolute spectral gap} of $A$ with respect to $d_\MM^q$, denoted by $\gamma_+(A,d_\MM^q)$, is the least constant $\gamma_+\in(0,\infty]$ such that
\begin{equation}
\forall \ x_1,\ldots,x_n,y_1,\ldots,y_n\in\MM, \qquad \sum_{i,j=1}^n \pi_i\pi_j d_\MM(x_i,y_j)^q \leq \gamma_+ \sum_{i,j=1}^n \pi_i a_{ij} d_\MM(x_i,y_j)^q.
\end{equation}
The terminology stems from the fact that $\gamma_+(A,|\cdot|^2) = (1-\max_{i=2,\ldots,n} |\lambda_i(A)|)^{-1}$, where $1=\lambda_1(A)\geq\lambda_2(A)\geq\cdots\geq\lambda_n(A)\geq-1$ are the eigenvalues of $A$. Nonlinear spectral gaps and nonlinear absolute spectral gaps are related via the following inequalities.

\begin{lemm} \label{lem:abs}
Fix $q\in[1,\infty)$, $n\in\N$ and $\pi\in\triangle^{n-1}$. For every $\pi$-reversible stochastic matrix $A\in M_n(\R)$ and every metric space $(\MM,d_\MM)$, we have
\begin{equation} \label{eq:abs}
2\gamma(A,d_\MM^q) \leq \gamma_+\Big(\frac{A+\msf{Id}_n}{2}, d_\MM^q\Big) \leq 2^{2q+1} \gamma(A,d_\MM^q).
\end{equation}
\end{lemm}

The elementary proof of Lemma \ref{lem:abs} can be found in \textcite[Lemma~2.3]{Nao14}. The pertinence of absolute spectral gaps in the ensuing discussion is that, in the case of uniformly convex spaces, they have a useful connection to vector-valued operator norms of adjacency matrices. This is manifested by the following proposition of \textcite[Lemma~6.6]{MN14}, whose proof relies on Pisier's martingale cotype inequality for $q$-uniformly convex spaces \parencite{Pis75}.

\begin{prop} \label{prop:bound||}
Fix $q\in[2,\infty)$ and let $(Y,\|\cdot\|_Y)$ be a $q$-uniformly convex normed space. Then, for every $n\in\N$, $\pi\in\triangle^{n-1}$ and every $\pi$-reversible stochastic matrix $C\in M_n(\R)$, we have
\begin{equation} \label{eq:bound||}
\|C\otimes\msf{Id}_Y\|_{L_q^{^0}(\pi;Y)\to L_q^{^0}(\pi;Y)} \leq \Big( 1- \frac{1}{(2^{q-1}-1)\msf{K}_q(Y)^q \gamma_+(C,\|\cdot\|_Y^q)}\Big)^{\frac{1}{q}}.
\end{equation}
\end{prop}

Proposition \ref{prop:bound||} is proven by \textcite{MN14} for the special case that $\pi$ is the uniform measure on $\{1,\ldots,n\}$ and $C$ is a symmetric stochastic matrix. The proof of the general statement presented here is similar to this special case and we thus omit it. Plugging the bound \eqref{eq:bound||} in \eqref{eq:66} for $C=\frac{A+\msf{Id}_n}{2}$, we finally deduce that
\begin{equation}
\begin{split}
\gamma(A,\|\cdot\|_X^p) \stackrel{\eqref{eq:66}\wedge\eqref{eq:bound||}}{\lesssim}_{_{\!\!\!\!\!\!\!\msf{M}_p(X), \msf{K}_q(Y)},p,q}  \log(\msf{c}_Y(X)&+1)  \gamma_+\Big(\frac{A+\msf{Id}_n}{2}, \|\cdot\|_Y^q\Big) \\ & \stackrel{\eqref{eq:abs}}{\lesssim_q} \log(\msf{c}_Y(X)+1)  \gamma(A,\|\cdot\|_Y^q). \  \hskip0.07\textwidth \Box
\end{split}
\end{equation}


\section{Geometric and algorithmic applications} \label{sec:applications}

In this final section, we present a selection of geometric and algorithmic applications of nonlinear spectral gaps (mostly without proofs) and related open questions.


\subsection{Nonembeddability of expanders into low-dimensional normed spaces} \label{sec:expanders}

Let $G=(V,E)$ be a $d$-regular graph on the vertex set $\{1,\ldots,n\}$. We shall denote by $A_G$ the normalized adjacency matrix of $G$, that is, the $n\times n$ symmetric stochastic matrix whose entries are given by $(A_G)_{ij} = \frac{{\bf 1}_{\{i,j\}\in E}}{d}$, where $i,j\in\{1,\ldots,n\}$. A sequence $\{G_n=(V_n,E_n)\}_{n=1}^\infty$ of $d$-regular graphs with $|V_n|\to\infty$ as $n\to\infty$ is called an expander graph sequence if $\sup_{n\in\N} \gamma(A_{G_n})<\infty$. The existence of regular expander graph sequences is a classical fact that can be proven via the probabilistic method (\cite{Pin73} and \cite{Bol88}), while deterministic constructions are notoriously more involved (see, e.g., the book of \cite{DSV03}). Embeddability properties of connected expanders viewed as metric spaces when equipped with the shortest path distance were first investigated by \textcite{LLR95} who, among other results, showed that if an $n$-vertex $d$-regular expander embeds with quadratic average distortion $D$ in a $k$-dimensional normed space, then $k\gtrsim (\log_d n)^2/D^2$. We shall now present the following (sharp) improvement of Linial, London and \mbox{Rabinovich's result due to \textcite{Nao17} as a consequence of the average John theorem.}

\begin{theo} \label{thm:expanders}
For every $q\in[1,\infty)$, there exists $c(q)\in(0,\infty)$ such that the following holds for every $\gamma,D\in[1,\infty)$. Let $G=(V,E)$ be a $d$-regular connected graph on $n$ vertices with $\gamma(A_G)\leq\gamma$ and let $(X,\|\cdot\|_X)$ be a normed space such that $(G,d_G)$ admits an embedding into $X$ with $q$-average distortion at most $D$. Then,
\begin{equation} \label{eq:expanders}
\mathrm{dim}(X) \geq n^{c(q)/\gamma D \log d}.
\end{equation}
\end{theo}

\begin{proof}
We shall first prove the case $q=1$. By the assumption, there exists a $D$-Lipschitz map\footnote{As is common, we shall identify the graph $G$ with its vertex set $V$ thus writing $f:G\to X$ rather than $f:V\to X$. Moreover, we will always denote by $d_G$ the shortest path distance on $V$.} $f:(G,d_G)\to (X,\|\cdot\|_X)$ satisfying the average lower bound
\begin{equation} \label{eq:x1}
\frac{1}{n^2} \sum_{u,v\in V} \|f(u)-f(v)\|_X \geq \frac{1}{n^2} \sum_{u,v\in V} d_G(u,v).
\end{equation}
Let $k=\mathrm{dim}(X)$. Applying Theorem \ref{thm:john} for the measure $\mu = \frac{1}{n} \sum_{u\in V} \delta_{f(u)}$ on $X$, we deduce that there exists a $O(\sqrt{\log k})$-Lipschitz function $h:(X,\|\cdot\|_X^{1/2}) \to \ell_2$ such that
\begin{equation} \label{eq:x2}
\frac{1}{n^2} \sum_{u,v\in V} \|h(f(u))-h(f(v))\|_{\ell_2}^2 \geq \frac{1}{n^2} \sum_{u,v\in V} \|f(u)-f(v)\|_X.
\end{equation}
Since the graph $G$ is a regular expander, inequality \eqref{eq:spegap} implies that
\begin{equation} \label{eq:x3}
\begin{split}
\frac{1}{n^2} \sum_{u,v\in V} \|h(f(u))-h(f(v))\|_{\ell_2}^2 & \leq \frac{2\gamma}{dn} \sum_{\{a,b\}\in E} \|h(f(a))-h(f(b))\|_{\ell_2}^2
\\ & \lesssim \frac{\gamma \log k}{dn} \sum_{\{a,b\}\in E} \|f(a)-f(b)\|_X \lesssim \gamma D\log k,
\end{split}
\end{equation}
where in the last two inequalities we used the Lipschitz conditions for $h$ and $f$. On the other hand, the graph $G$ is $d$-regular and therefore, for any fixed $u\in V$ there exist at least $\frac{n}{2}$ vertices $v\in V$ such that $d_G(u,v)\geq \lfloor\log_d(n/2)\rfloor$. Hence, we have
\begin{equation} \label{eq:x4}
\frac{1}{n^2} \sum_{u,v\in V} d_G(u,v) \gtrsim \log_d n,
\end{equation}
which, combined with \eqref{eq:x1}, \eqref{eq:x2} and \eqref{eq:x3}, implies that
\begin{equation}
\log_d n \lesssim \gamma D \log k,
\end{equation}
thus completing the proof of \eqref{eq:expanders} for $q=1$. To address the general case $q\geq1$, we need a slight modification of this argument. It is a formal consequence of \textcite[Proposition~6]{Nao21} and Theorem \ref{thm:john}, that for any $q\geq1$, the $\tfrac{1}{2}$-snowflake of any finite-dimensional normed space $X$ embeds into $\ell_2$ with $(2q)$-average distortion at most $C(q)\sqrt{\log(\mathrm{dim}(X)+1)}$. Considering a $O(\sqrt{\log k})$-Lipschitz embedding satisfying the analogue of \eqref{eq:x2} with power $2q$ instead of the embedding $h$ and repeating the above argument completes the proof of \eqref{eq:expanders} for general $q\geq1$.
\end{proof}

A few historical comments are in order. Due to the existence of regular expander graph sequences, Theorem \ref{thm:expanders} implies that for arbitrarily large $n$, there exists an $n$-point metric space $(\MM_n,d_{\MM_n})$ such that if $\MM_n$ admits an embedding with bi-Lipschitz distortion $D$ into a finite-dimensional normed space $X$, then $\mathrm{dim}(X)\geq n^{c/D}$ for some universal constant $c\in(0,\infty)$. Therefore, Theorem \ref{thm:expanders} provides a negative answer to the question of \textcite{JL84} discussed in Section \ref{subsec:history}. A different negative answer to this question had been given in important work of \textcite{Mat96}, who devised an ingenious random family of metric spaces and showed that they satisfy this property using input from real algebraic geometry. It is worth mentioning that a precursor of Theorem \ref{thm:expanders} is a result of \textcite[Proposition~4.1]{LMN05}, who showed that if an $n$-vertex regular expander embeds in $\ell_\infty^d$ with bi-Lipschitz distortion at most $D$, \mbox{then $d\geq n^{c/D}$ for some universal constant $c\in(0,\infty)$.}

Quantitatively, Theorem \ref{thm:expanders} provides a sharp relation between the dimension of the target space $X$, the number of vertices of $G$ and the distortion $D$. Indeed, a classical theorem of \textcite{JLS87} asserts that for every $n\in\N$ and $D\geq1$, any $n$-point metric space $\MM$ admits a bi-Lipschitz embedding with distortion at most $D$ into some $d$-dimensional normed space $X$, where $d\lesssim_D n^{C/D}$ for some universal constant $C\in(0,\infty)$. This result was later refined by \textcite{Mat92}, who showed that one can always take $X=\ell_\infty^d$ as a target space in this statement.

\begin{rema} \label{rem:sharp3}
The optimality of Theorem \ref{thm:expanders} which follows from the works of \textcite{JLS87} and \textcite{Mat92} immediately implies that the  $O(\sqrt{\log\mathrm{dim}(X)})$ upper bound for the average distortion in Theorem \ref{thm:john} is sharp. Indeed, suppose that the $\tfrac{1}{2}$-snowflake of $X=\ell_\infty^d$ admitted an embedding into $\ell_2$ with quadratic average distortion $o(\sqrt{\log d})$. Then, the proof of Theorem \ref{thm:expanders} would show that if an $n$-vertex expander \mbox{embeds with bi-Lipschitz distortion $D$ in $\ell_\infty^d$, then}
\begin{equation}
\log n = o(D \log d).
\end{equation}
However, this inequality contradicts the embedding theorem of \textcite{Mat92}.
\end{rema}

Following the terminology of \textcite{Nao18}, we say that an infinite-dimensional Banach space $(X,\|\cdot\|_X)$ admits (quadratic) average dimension reduction with distortion \mbox{$D\in(1,\infty)$} if for any $n\in\N$ there exists $k_n=k_n^D(X)\in\N$ satisfying
\begin{equation}
\lim_{n\to\infty} \frac{\log k_n}{\log n} = 0
\end{equation}
such that the following condition holds. For any $n$ points $x_1,\ldots,x_n\in X$, there exists a subspace $F=F(x_1,\ldots,x_n)$ of $X$ with $\mathrm{dim}F\leq k_n$ and points $y_1,\ldots,y_n\in F$ satisfying $\|y_i-y_j\|_X\leq D \|x_i-x_j\|_X$ for every $i,j\in\{1,\ldots,n\}$ and
\begin{equation}
\frac{1}{n^2} \sum_{i,j=1}^n \|y_i-y_j\|_X^2 \geq \frac{1}{n^2} \sum_{i,j=1}^n \|x_i-x_j\|_X^2.
\end{equation}
As every finite metric space embeds isometrically in $\ell_\infty$, the aforementioned result of \textcite{Mat96} (or Theorem \ref{thm:expanders}) implies that $\ell_\infty$ does not admit average dimension reduction with any distortion $D>1$. The following tantalizing question remains open.

\begin{question} \label{q:l1}
Does $\ell_1$ admit average dimension reduction with any distortion $D>1$?
\end{question}

We note that the bi-Lipschitz analogue of Question \ref{q:l1} is answered by a famous theorem of \textcite{BC05} (see also \textcite{LN04} for a different influential proof) who showed that for arbitrarily large $n$ and $D>1$ there exists an $n$-point subset of $\ell_1$ which does not admit a bi-Lipschitz embedding into any subspace of $\ell_1$ of dimension at most $n^{c/D^2}$, where $c\in(0,\infty)$ is a universal constant.


\subsection{Average distortion embeddings of $\ell_p$ into $\ell_2$} 

In Theorem \ref{thm:john}, it was established that any finite-dimensional normed space $(X,\|\cdot\|_X)$ admits an embedding into $\ell_2$ with quadratic average distortion $O(\sqrt{\log(\mathrm{dim}(X)+1)})$ via the nonlinear spectral gap inequality \eqref{eq:estimate}. As explained in Remark \ref{rem:sharp3}, this estimate for the quadratic average distortion is asymptotically optimal yet, quite surprisingly, there exist many non-Hilbertian normed spaces which embed with constant quadratic average distortion in $\ell_2$. The following result is the main theorem of \textcite{Nao14}.

\begin{theo} \label{thm:lp}
There exists $C\in(0,\infty)$ such that for any $p\in(2,\infty)$, the normed space $\ell_p$ admits an embedding into $\ell_2$ with quadratic average distortion $Cp$.
\end{theo}

Theorem \ref{thm:lp} is established in \textcite{Nao14} via the nonlinear spectral gap inequality
\begin{equation} \label{eq:lp}
\forall \ p>2, \qquad \gamma(A,\|\cdot\|_{\ell_p}^2) \lesssim \frac{p^2}{1-\lambda_2(A)},
\end{equation}
which holds for any $\pi\in\triangle^{n-1}$ and any $\pi$-reversible stochastic matrix $A\in M_n(\R)$, and the duality principle of Theorem \ref{thm:duality}. Once again, \eqref{eq:lp} is proven in \textcite{Nao14} in the special case that $\pi$ is the uniform measure on $\{1,\ldots,n\}$ and $A$ is a symmetric stochastic matrix. The proof of the more general statement presented here (which is equivalent to Theorem \ref{thm:lp}) is identical. In \textcite{Nao14}, Theorem \ref{thm:lp} and \eqref{eq:lp} were used to give new lower bounds for the $\ell_p$-distortion of random connected $d$-regular graphs, Ramanujan graphs and abelian Alon--Roichman graphs, improving earlier results of \textcite{Mat97}. It is worth pointing out that \eqref{eq:lp} is no longer valid when $p\in[1,2)$.


\subsection{Expanders with respect to Banach spaces} \label{sec:super-expanders}

Combinatorial expanders are ubiquitous geometric objects whose metric structure is notoriously incompatible with Euclidean geometry. Nonlinear spectral gaps allow us to analyze non-Euclidean analogues of these exotic metrics.  Let $(\MM,d_\MM)$ be a metric space. A sequence $\{G_n\}_{n=1}^\infty$ of $d$-regular graphs with $|G_n|\to\infty$ is called an expander graph sequence with respect to $\MM$ if $\sup_{n\in\N}\gamma(A_{G_n}, d_\MM^2) <\infty$. If such graphs exist, we say that $(\MM,d_\MM)$ admits a sequence of $d$-regular expanders. The following influential observation on the embeddability of expanders is essentially due to \textcite{Mat97}.

\begin{prop} \label{prop:matousek}
Let $(\MM,d_\MM)$ be a metric space and fix $\gamma,q\in(0,\infty)$. Suppose that $G=(V,E)$ is a $d$-regular connected graph on $n$ vertices with $\gamma(A_G,d_\MM^q)\leq \gamma$. If $(G,d_G)$ embeds into $(\MM,d_\MM)$ with $q$-average distortion at most $D$, then $D\gtrsim \frac{\log_d n}{\gamma^{1/q}}$.
\end{prop}

\begin{proof}
By the assumption, there exists $\sigma\in(0,\infty)$ and a $\sigma D$-Lipschitz map \mbox{$f:(G,d_G)\to(\MM,d_\MM)$ satisfying the average lower bound}
\begin{equation}
\frac{1}{n^2} \sum_{u,v\in V} d_\MM\big(f(u),f(v)\big)^q \geq \frac{\sigma^q}{n^2} \sum_{u,v\in V} d_G(u,v)^q \gtrsim \sigma^q(\log_d n)^q,
\end{equation}
where the second inequality follows from \eqref{eq:x4} and Jensen's inequality. On the other hand, by the definition \eqref{eq:nsg} of $\gamma(A,d_\MM^q)$, we have
\begin{equation}
\frac{1}{n^2} \sum_{u,v\in V} d_\MM\big(f(u),f(v)\big)^q \leq \frac{2\gamma}{dn} \sum_{\{a,b\}\in E} d_\MM\big( f(a),f(b)\big)^q \leq \gamma \sigma^q D^q,
\end{equation}
where the last inequality follows from the Lipschitz condition for $f$. Rearranging, we deduce the desired lower bound for the $q$-average distortion $D$.
\end{proof}

Deciding whether a given non-Euclidean metric space admits a sequence of expanders is a notoriously difficult problem in metric geometry, even when specified to normed spaces. By Proposition \ref{prop:matousek}, it is clear that if there exists a sequence $\{G_n\}_{n=1}^\infty$ of regular expanders with respect to a normed space $(X,\|\cdot\|_X)$, then $X$ cannot contain subspaces uniformly isomorphic to $\{\ell_\infty^m\}_{m=1}^\infty$ as it would then bi-Lipschitzly contain all finite metric spaces with uniform distortion. Normed spaces which do not uniformly contain $\{\ell_\infty^m\}_{m=1}^\infty$ are said to have finite cotype in Banach space theory jargon \parencite{Mau03}. Strikingly, this is the only known necessary condition for a normed space to admit an expander graph sequence and the following general question remains open.

\begin{question} \label{q:cotype}
Is every combinatorial expander also an expander with respect to any normed space of finite cotype?
\end{question}

Such implications, asserting that a classical spectral gap implies a nonlinear spectral gap, are currently only known for substantially smaller classes of normed spaces from works of \textcite{Mat97}, \textcite{Oza04}, \textcite{Pis10} and \textcite{NS11}. It is worth mentioning that even the following question, which is formally weaker than Question \ref{q:cotype} in view of Proposition \ref{prop:matousek}, remains open.

\begin{question} \label{q:bilip}
Does there exist a sequence of finite metric spaces $\{(\MM_n,d_{\MM_n})\}_{n=1}^\infty$ with $|\MM_n|\to\infty$ as $n\to\infty$ such that for any normed space $(X,\|\cdot\|_X)$ of finite cotype, the bi-Lipschitz distortion required to embed $\MM_n$ into $X$ satisfies $\msf{c}_X(\MM_n) \gtrsim_X \log|\MM_n|$?
\end{question}

A positive answer to Question \ref{q:bilip} would imply a striking dichotomy in the embeddability of finite metric spaces into infinite-dimensional normed spaces. If such a normed space $X$ does not have finite cotype, then it bi-Lipschitzly contains every finite metric space with distortion $1+\e$ for any $\e>0$ \parencite{Mau03}. On the other hand, if $X$ is an arbitrary infinite-dimensional space, then any finite metric space $\MM$ admits a bi-Lipschitz embedding into $X$ with distortion $O(\log|\MM|)$ by the theorems of \textcite{Dvo60} and \textcite{Bou85}. A positive answer to Question \ref{q:bilip} would imply that this bound is always optimal under the (necessary) assumption that $X$ has finite cotype.

In regard to Question \ref{q:cotype}, even the \emph{existence} of a sequence $\{G_n\}_{n=1}^\infty$ of regular graphs which are expanders with respect to any space of finite cotype remains unknown. The strongest available result in this direction is the following profound theorem of \textcite{Laf08}, whose proof is an ingenious combination of algebraic and vector-valued harmonic analytic methods. We say that a normed space $(X,\|\cdot\|_X)$ has nontrivial type if $X$ does not contain subspaces uniformly isomorphic to $\{\ell_1^m\}_{m=1}^\infty$ \parencite{Mau03}. Any space of nontrivial type has finite cotype, but the converse is not true (e.g.~for $\ell_1$).

\begin{theo} \label{thm:lafforgue}
There exists a sequence of regular graphs $\{G_n\}_{n=1}^\infty$ which is an expander graph sequence with respect to any normed space of non-trivial type.
\end{theo}

Lafforgue’s graphs can be obtained as Cayley graphs of finite quotients of co-compact lattices in $SL_3(\mathbb{Q}_p)$, where $p$ is a prime and $\mathbb{Q}_p$ is the field of $p$-adic rationals.

A completely different construction of a sequence of regular graphs which are expanders with respect to a large family of norms was presented in work of \textcite{MN14}. Theirs is a vector-valued adaptation of the zig-zag product construction of \textcite{RVW02} and the resulting graphs are expanders with respect to any normed space which admits an equivalent uniformly convex norm. Clearly any such space has nontrivial type but the converse is not true \parencite{Pis75b}. While we will not outline the argument of \textcite{MN14}, it is worth pointing out that it consists of a novel construction of a base graph along with an adaptation of the zig-zag iteration of \textcite{RVW02}. The necessity of the uniform convexity assumption in this argument stems from this iteration procedure. On the other hand, the construction of the base graph (which was straightforward in the case of combinatorial expanders) has raised influential questions in vector-valued harmonic analysis that led to investigations of independent interest (\cite{MN14}; \cite{EI20, EI21}).


\subsection{Expanders with respect to Alexandrov spaces}

A complete geodesic metric space $(\MM,d_\MM)$ is an Alexandrov space of nonpositive curvature (or a CAT(0) space) if for any quadruple of points $x,y,z,m\in \MM$ such that $m$ is a metric midpoint of $x$ and $y$, that is, $d_\MM(m,x)=d_\MM(m,y)=\tfrac{1}{2}d_\MM(x,y)$, we have
\begin{equation}
d_\MM(z,m)^2 \leq \frac{1}{2}d_\MM(z,x)^2+\frac{1}{2}d_\MM(z,y)^2 - \frac{1}{4}d_\MM(x,y)^2.
\end{equation}
If the reverse inequality holds true for any such quadruple $x,y,z,m\in\MM$, then $\MM$ is an Alexandrov space of nonnegative curvature. Alexandrov spaces of nonpositive (respectively nonnegative) curvature are (potentially singular) metric spaces which generalize Riemannian manifolds with nonpositive (resp.~nonnegative) sectional curvature.

An argument of \textcite{Wan98} shows that any regular combinatorial expander is also an expander with respect to any Hilbert manifold with a CAT(0) Riemannian metric (see also \cite[Corollary~4.10]{NS11}). The first systematic study of expanders with respect to (non-smooth) Alexandrov spaces of nonpositive curvature was undertaken by \textcite{MN15}, who showed the following theorem.

\begin{theo} \label{thm:cat0}
There exists a CAT(0) space $(\MM,d_\MM)$ and a sequence $\{G_n\}_{n=1}^\infty$ of 3-regular graphs such that $\sup_{n\in\N} \gamma(A_{G_n},d_\MM^2)<\infty$, yet a random $d$-regular graph ${\bf G}$ on $n$ vertices satisfies $\gamma(A_{\bf G}, d_\MM^2) \gtrsim (\log_d n)^2$ with probability $1-o_n(1)$ as $n\to\infty$.
\end{theo}

Theorem \ref{thm:cat0} reveals a striking difference between nonlinear spectral gaps with respect to Alexandrov spaces of nonpositive curvature and classical spectral gaps, as a random $d$-regular graph on $n$-vertices is a combinatorial expander with probability $1-o_n(1)$ as $n\to\infty$ for any fixed $d\in\N$ \parencite{Bol88}. The following question remains open.

\begin{question} \label{q:cat0}
Does every CAT(0) space admit a sequence of regular expanders? More ambitiously, does there exists a sequence of $O(1)$-regular graphs $\{G_n\}_{n=1}^\infty$ with $|G_n|\to\infty$ such that $\sup_{n\in\N}\gamma(A_{G_n},d_\MM^2)<\infty$ for every CAT(0) space $(\MM,d_\MM)$?
\end{question}

A positive answer to the stronger statement in Question \ref{q:cat0} would imply (in view of Proposition \ref{prop:matousek}) the existence of arbitrarily large finite metric spaces requiring logarithmic distortion to be embedded in any Alexandrov space of nonpositive curvature. The following question was asked by \textcite{EMN19}.

\begin{question} \label{q:bilcat}
Does there exist a sequence of finite metric spaces $\{(\MM_n,d_{\MM_n})\}_{n=1}^\infty$ with $|\MM_n|\to\infty$ as $n\to\infty$ such that for any CAT(0) space $(\NN,\|\cdot\|_\NN)$, the bi-Lipschitz distortion required to embed $\MM_n$ into $\NN$ satisfies $\msf{c}_\NN(\MM_n) \gtrsim_\NN \log|\MM_n|$?
\end{question}

In the dual nonnegative curvature regime, the analogue of Question \ref{q:cat0} was answered by \textcite{ANN18}, who showed that there exists an Alexandrov space of nonnegative curvature which does not admit any sequence of regular expanders. Moreover, they asked the following dual to Question \ref{q:bilcat}.

\begin{question} \label{q:pos}
Does there exist an Alexandrov space of nonnegative curvature $(\NN,d_\NN)$ such that any finite metric space $(\MM,d_\MM)$ embeds into $\NN$ with bi-Lipschitz distortion $\msf{c}_\NN(\MM)\lesssim \sqrt{\log|\MM|}$?
\end{question}

In their paper, they specifically asked Question \ref{q:pos} for the concrete Alexandrov space $\mc{P}_2(\R^3)$, which is the space of all Borel probability measures $\mu$ on $\R^3$ satisfying $\int_{\R^3} \|x\|_{\ell_2^3}^2 \,\diff \mu(x)<\infty$ equipped with the Wasserstein $\msf{W}_2$-distance.


\subsection{Coarse non-universality}

Let $(\MM,d_\MM)$ and $(\NN,d_\NN)$ be two metric spaces and $\omega,\Omega:[0,\infty)\to[0,\infty)$ two moduli satisfying $\omega\leq\Omega$ pointwise and $\lim_{t\to\infty} \omega(t)=\infty$. A mapping $f:\MM\to\NN$ is a coarse embedding with lower and upper moduli $\omega$ and $\Omega$ respectively if
\begin{equation}
\forall \ x,y\in\MM, \qquad \omega\big(d_\MM(x,y)\big) \leq d_\NN\big(f(x),f(y)\big) \leq \Omega\big(d_\MM(x,y)\big).
\end{equation}
A family of metric spaces $\{(\MM_\alpha,d_{\MM_\alpha})\}_\alpha$ is said to embed equi-coarsely into a metric space $(\NN,d_\NN)$ if there exist two moduli $\omega,\Omega:[0,\infty)\to[0,\infty)$ satisfying $\omega\leq\Omega$ pointwise and $\lim_{t\to\infty} \omega(t)=\infty$ and a family of coarse embeddings $\{f_\alpha:\MM_\alpha\to\NN\}_\alpha$ with lower and upper moduli $\omega$ and $\Omega$. The pertinence of nonlinear spectral gaps in coarse geometry stems from the\mbox{ following influential observation of \textcite{Gro00,Gro03}.}

\begin{prop} \label{prop:gromov}
Fix $d\in\N$, $p\in(0,\infty)$ and let $(\MM,d_\MM)$ be a metric space. Suppose that $\{G_n=(V_n,E_n)\}_{n=1}^\infty$ is a sequence of connected $d$-regular graphs with \mbox{$|V_n|\to\infty$} and $\sup_{n\in\N} \gamma(A_{G_n},d_\MM^p)<\infty$. Then, the family of graphs $\{(G_n,d_{G_n})\}_{n=1}^\infty$ equipped with their shortest path distances does not equi-coarsely embed into $\MM$.
\end{prop}

\begin{proof}
Let $\gamma\eqdef\sup_{n\in\N} \gamma(A_{G_n},d_\MM^p)<\infty$. Suppose that there exist two moduli $\omega,\Omega:[0,\infty)\to[0,\infty)$ with $\lim_{t\to\infty} \omega(t)=\infty$ and mappings $f_n:V_n\to \MM$ with
\begin{equation} \label{eq:coarse-assume}
\forall \ x,y\in V_n, \qquad \omega\big(d_{G_n}(x,y)\big) \leq d_\MM\big(f(x),f(y)\big) \leq \Omega\big(d_{G_n}(x,y)\big).
\end{equation}
By definition of nonlinear spectral gaps, we have
\begin{equation}
\frac{1}{|V_n|^2} \sum_{u,v\in V_n} d_\MM\big(f_n(u),f_n(v)\big)^p \leq \frac{2\gamma}{d|V_n|} \sum_{\{a,b\}\in E_n}d_\MM\big(f_n(a),f_n(b)\big)^p.
\end{equation}
Moreover, using the upper modulus, we get
\begin{equation}
\frac{2\gamma}{d|V_n|} \sum_{\{a,b\}\in E_n}d_\MM\big(f_n(a),f_n(b)\big)^p \stackrel{\eqref{eq:coarse-assume}}{\leq} \frac{2\gamma}{d|V_n|} \sum_{\{a,b\}\in E_n}\Omega(1)^p =\gamma \Omega(1)^p.
\end{equation}
Finally, as each graph $G_n$ is $d$-regular, for any $u\in V_n$ there exist at least $\frac{|V_n|}{2}$ vertices $v\in V_n$ such that $d_{G_n}(u,v)\geq \lfloor\log_d(|V_n|/2)\rfloor$. Thus, the lower modulus gives
\begin{equation*}
\frac{1}{|V_n|^2} \sum_{u,v\in V_n} d_\MM\big(f_n(u),f_n(v)\big)^p \stackrel{\eqref{eq:coarse-assume}}{\geq} \frac{1}{|V_n|^2} \sum_{u,v\in V_n} \omega\big(d_{G_n}(u,v)\big)^p \geq \frac{\omega(\lfloor\log_d(|V_n|/2)\rfloor)^p}{2}.
\end{equation*}
Combining all the above, we deduce that
\begin{equation}
\forall \ n\in\N, \qquad \omega(\lfloor\log_d(|V_n|/2)\rfloor)^p \leq 2\gamma\Omega(1)^p,
\end{equation}
which clearly contradicts the coarse condition $\lim_{t\to\infty} \omega(t)=\infty$.
\end{proof}

An important consequence of Gromov's observation is that if $(\MM,d_\MM)$ admits a sequence of regular expanders $\{G_n\}_{n=1}^\infty$, then there exists a metric space (e.g.~the disjoint union $\bigsqcup_{n\geq1} (G_n,d_{G_n})$) which does not admit a coarse embedding into $\MM$. Consequently, the mere existence of combinatorial expanders implies that Hilbert spaces are not coarsely universal which is a well-known theorem of \textcite{DGLY02}. Moreover, Lafforgue's Theorem \ref{thm:lafforgue} implies the existence of a metric space which does not admit a coarse embedding into any Banach space of non-trivial type. The coarse non-universality of this class was previously established in work of \textcite{MN08} by proving that Banach spaces of non-trivial type with cotype $q$ have \emph{sharp} metric cotype $q$. Understanding whether every Banach space of cotype $q$ has sharp metric cotype $q$ is the central open problem in the theory of metric cotype of Banach spaces; see \textcite{GMN11} for the best known results to date. If this was the case, then the following (currently open) question on coarse embeddings would have a negative answer.

\begin{question} \label{q:coarse-cotype}
Does every separable metric space embed coarsely into some Banach space of finite cotype?
\end{question}

It follows from Proposition \ref{prop:gromov} that a negative answer to Question \ref{q:coarse-cotype} would also be a consequence of the existence of a sequence of regular graphs $\{G_n\}_{n=1}^\infty$ which are expanders with respect to any normed space of finite cotype simultaneously, let alone from a positive answer to the much stronger Question \ref{q:cotype}. 

Despite the fact that Question \ref{q:cat0} on the existence of expanders with respect to Alexandrov spaces of nonpositive curvature remains open, the coarse non-universality of this class was established by \textcite{EMN19}, thus answering a question raised by \textcite{Gro93}. The main technical contribution of this work is the proof that every CAT(0) space $\MM$ has sharp metric cotype 2 which formally implies that $\ell_q$ does not admit a coarse embedding in $\MM$ for any $q>2$. In contrast to this result, the very surprising fact that there exist coarsely universal Alexandrov spaces of nonnegative curvature was proven by \textcite{ANN18}.


\subsection{Approximate nearest neighbor search}

Fix a parameter $c>1$. The $c$-Approximate Nearest Neighbor Search problem is defined as follows. Given an $n$-point dataset $\mc{P}$ in some metric space $(\MM,d_\MM)$, we want to build a data structure\footnote{For the purposes of this survey, a data structure of size $M$ is an array $A[1\ldots M]$ of numbers (the ``memory'') along with an algorithm which, given a point $q\in \MM$, returns a point $\hat{p}\in\mc{P}$.} that, given any query point $q\in\MM$, returns a point $\hat{p}\in\mc{P}$ with $d_\MM(q,\hat{p})\leq c\min_{p\in\mc{P}}d_\MM(q,p)$. In practice, this problem can be reduced to its ``decision version'' (see \cite{HPIM12}), which is the $c$-Approximate \emph{Near} Neighbor Search ($c$-ANN) problem at a \emph{pre-fixed} distance scale $r>0$. In the $c$-ANN problem at scale $r$, we are again given an $n$-point dataset $\mc{P}$ in some metric space $(\MM,d_\MM)$ and we want to build a data structure that, given any query point $q\in\MM$ for which there exists a point $p^\ast\in \mc{P}$ with $d_\MM(q,p^\ast)\leq r$, returns a point $\hat{p}\in\mc{P}$ with $d_\MM(q,\hat{p})\leq cr$. The main parameters to optimize are the \emph{space} the data structure occupies and the \emph{time} it takes to answer a query. A majority of the research conducted on this problem has focused on $d$-dimensional normed spaces rather than general metric spaces and moreover most of the algorithms in the literature are \emph{randomized} in the sense that they return a random point $\hat{p}$ satisfying $d_\MM(q,\hat{p})\leq cr$ with probability at least $1-\delta$ for some pre-fixed confidence parameter $\delta\in(0,1)$.

The first approaches to the $c$-ANN problem for $d$-dimensional norms produced data-independent data structures, in which the memory cells accessed by the algorithm do not depend on the dataset $\mc{P}$ but only on the query point $q$. In particular, building such data structures via (oblivious) metric dimension reduction has been used with great success for the Hilbert space $\ell_2^d$ (\cite{IM98}; \cite{HPIM12}), the  hypercube $\{-1,1\}^d$ equipped with the Hamming distance \parencite{KOR00} and spaces which (effectively) embed in them \parencite{AIK09, Ngu14}. While dimension reduction techniques yield data structures with polynomial space for these norms, these results are often far from practical due to the large degree of said polynomial. To overcome this barrier, \textcite{IM98} introduced an influential technique called \emph{Locality-Sensitive Hashing} (LSH) relying on (data-independent) randomized space  partitions. Somewhat informally, a distribution $\mc{D}$ over a family of partitions of $\MM$ is called \emph{sensitive} at scale $r$ up to error $c$ if any two points at distance at most $r$ are $\mc{D}$-likely to belong in the same cluster of the partition and any two points at distance at least $cr$ are $\mc{D}$-unlikely to do so (where the implicit probabilities affect the space and time requirements of the data structure). As shown by \textcite{IM98}, a (computationally efficient) sensitive distribution over random partitions can serve as  a pre-filter for the dataset $\mc{P}$ as the query point $q$ is very likely to be indistinguishable from its near neighbors but is unlikely to collide with points $p$ having $d_\MM(q,p)>cr$. Using LSH, they were able to improve the space requirements over the existing $c$-ANN algorithms to almost linear for large enough accuracy parameters $c>1$. We refer to the thorough survey of \textcite{AIR18} for a detailed account of these and other contributions on the $c$-ANN problem and further references.

Despite these advances towards the $c$-ANN problem, researchers proved strong impossibility results \parencite{MNP07, OWZ14} for the existence of data-independent data structures arising from LSH, thus creating the necessity for the development of efficient data-\emph{dependent} algorithms. Historically, the first such result was proven by \textcite{Ind01} for $\ell_\infty^d$. In recent years, this approach has gained a lot of momentum, especially in view of the works of \textcite{AIK09} for the Ulam metric, \textcite{AINR14, AR15} for $\ell_2^d$ and \textcite{ANNRW17} for 1-symmetric norms. A  breakthrough in this direction was presented in the work of \textcite{ANNRW18b} who showed the following theorem for \emph{general} $d$-dimensional normed spaces. It is worth emphasizing that their result does not a priori give any bound on the running time of the algorithm, it just restricts the number of memory locations the data structure is allowed to probe.

\begin{theo} \label{thm:annrw}
Fix $\e\in(0,1)$ and let $X$ be a $d$-dimensional normed space. There exists a randomized data structure for \mbox{$O\big(\frac{\log d}{\e^2}\big)$-ANN over $X$ with the following properties:}
\begin{enumerate}
\item[$\bullet$] The space used by the data structure is $n^{1+\e}\cdot d^{O(1)}$;
\item[$\bullet$] The query procedure probes $n^\e\cdot d^{O(1)}$ words in memory.
\end{enumerate}
\end{theo}

In order to prove Theorem \ref{thm:annrw}, the authors introduced a geometric parameter called the cutting modulus $\Xi(\MM,\e)$ associated to a metric space $(\MM,d_\MM)$ and a parameter $\e\in(0,1)$, which governs the following data-dependent partitioning scheme:~every finite dataset in $\MM$ either has a subset of proportional size (measured appropriately) which is contained in a ball of radius $\Xi(\MM,\e)$ or admits a cut which is $\e$-sparse. Relying on this notion, \mbox{they were able to show the following general partitioning theorem.}
 
\begin{theo} \label{thm:partitioning}
Let $(\MM,d_\MM)$ be a finite metric space and fix $n\in\N$, $\e\in(0,\tfrac{1}{4})$. There exists a  collection $\mc{C}$ of subsets of $\MM$ with $\log|\mc{C}| = O\big(\log|\MM|\log(\log|\MM|/\e)\big)$ such that for any $n$-point dataset $\mc{P}$ in $\MM$, we have one of the following two properties:
\begin{enumerate}
\item[$\bullet$] Either there exists $x_0\in\MM$ and $R\leq\Xi(\MM,\e)$ such that $|\mc{P}\cap B_\MM(x_0,R)|\geq\tfrac{n}{50}$, or
\item[$\bullet$] There exists a subcollection $\{S_1,\ldots,S_m\} \subseteq \mc{C}$ such that
\begin{equation}
\forall \ i\in\{1,\ldots,m\}, \qquad \frac{n}{50} \leq |S_i\cap\mc{P}| \leq \frac{49n}{50}
\end{equation}
and for every $x,y\in\MM$ with $d_\MM(x,y)\leq1$, we have
\begin{equation}
\#\big\{i\in\{1,\ldots,m\}: \ {\bf 1}_{S_i}(x)\neq {\bf 1}_{S_i}(y)\big\} \leq 50\e m.
\end{equation}
\end{enumerate} 
\end{theo}

Theorem \ref{thm:partitioning} suggests a very natural LSH with approximation $O(\Xi(\MM,\e))$ since, at each step of the algorithm, we either have a dense ball of radius $\Xi(\MM,\e)$ or we have a collection of subsets with a distribution that decreases the size of the dataset and rarely splits the query from its nearby points in the dataset. The relevance of those results with the subject of this survey stems from the fact that Theorem \ref{thm:estimate} implies that if $X$ is a normed space, then the cutting modulus satisfies
\begin{equation} \label{eq:bound-cut}
\Xi(X,\e) \lesssim \frac{\log (\mathrm{dim}(X)+1)}{\e^2}.
\end{equation}
The main idea of the proof of \eqref{eq:bound-cut} is to apply \eqref{eq:estimate} to the adjacency matrices of geometric graphs associated to finite subsets of $X$. If such a graph does not have a subset of proportional size (with respect to the underlying stationary measure) contained in a ball of radius $\Omega(\log(\mathrm{dim}(X)))$, then  the nonlinear spectral gap inequality \eqref{eq:estimate} implies that it also cannot have large (classical) spectral gap and thus admits a sparse cut by Cheeger's inequality. A combination of Theorem \ref{thm:partitioning} and \eqref{eq:bound-cut} implies Theorem \ref{thm:annrw}. We refer to the work of \textcite{ANNRW18b} for the precise definition of the cutting modulus and the proofs of these results.

In their follow-up work, \textcite{ANNRW18a} proved the existence of data structures for ANN over $d$-dimensional normed spaces with slightly worse (but still subpolynomial) approximation and reasonable bounds for the running time. Some of their results still rely on elements of the theory of nonlinear spectral gaps, whereas others use the existence of a remarkable uniform homeomorphism between spheres of Banach spaces which originates in the resolution of the distortion problem by \textcite{OS94}. 


\bigskip

\printbibliography

\end{document}
